%% file: sample.tex
\documentclass[EJP, preprint]{ejpecp} 

\usepackage{tikz}

\SHORTTITLE{Explicit two-sided estimates and exact boundary regularity}

\TITLE{Explicit two-sided estimates and exact boundary regularity for the kinetic Fokker-Planck equation}  

\AUTHORS{%
  Mouad~Ramil\footnote{Univ Rennes, Inria, IRMAR [(Institut de Recherche Mathematique de Rennes)] - UMR 6625,
F-35000 Rennes, France.
    \EMAIL{mouad.ramil@inria.fr}}
  }


\AMSSUBJ{60J70 ; 35P15} 

\SUBMITTED{February 19, 2025} 
\ACCEPTED{} 

\ARXIVID{2209.02378} 
\HALID{hal-04437493} 

\VOLUME{0}
\YEAR{2025}
\PAPERNUM{0}
\DOI{10.1214/YY-TN}

\ABSTRACT{\textcolor{black}{Consider the transition density $\mathrm{p}^D_t(x,y)$ of the underdamped Langevin process absorbed outside of the cylindrical domain $D:=(0,1)\times\mathbb{R}$. We prove explicit and sharp \emph{two-sided estimates} satisfied by $\mathrm{p}^D_t$ on $D\times D$. We also establish the $C^\infty$ regularity of $\mathrm{p}^D$ on $\mathbb{R}_+^*\times(\overline{D}\setminus\Gamma^0)\times(\overline{D}\setminus\Gamma^0)$ where $\Gamma^0$ corresponds to the grazing boundary points. It is also shown to be $(1/6,1/2)$ H\"older continuous in the $(q,p)$ coordinates in a neighborhood of $\Gamma^0$. Finally, the long-time asymptotics of the absorbed transition density $\mathrm{p}^D_t(x,y)$ is obtained allowing to prove the uniform conditional ergodicity of the absorbed semigroup.}}

\newcommand\numberthis{\addtocounter{equation}{1}\tag{\theequation}} 
\hypersetup{
    colorlinks,
    linkcolor={black!50!black},
    citecolor={red}
}

\begin{document}

\section{Introduction} 

Let $(q_t,p_t)_{t\geq0}$ be the \textit{underdamped Langevin} process in $\mathbb{R}\times\mathbb{R}$ satisfying the following stochastic differential equation:
\begin{equation}\label{eq:Langevin_intro}
  \left\{
    \begin{aligned}
        &\mathrm{d}q_t=p_t \mathrm{d}t , \\
        &\mathrm{d}p_t=\textcolor{black}{F(q_t)}\mathrm{d}t-\gamma p_t\mathrm{d}t+\sigma\mathrm{d}B_t,
    \end{aligned}
\right.  
\end{equation} 
where $\gamma>0$, $\sigma=\sqrt{2\gamma k_BT}$, $k_B>0$ is the Boltzmann constant, $T$ is the fixed medium temperature, \textcolor{black}{$F$ is the force applied} and $(B_t)_{t\geq0}$ is a one-dimensional Brownian process. The underdamped Langevin process is used in statistical physics to model the evolution of a thermostated molecular system. In the above two-dimensional case, the particle is described at all time $t\geq0$ by its position $q_t\in\mathbb{R}$ and velocity $p_t\in\mathbb{R}$. It is also subjected to a force $F(q)$ and to collisions modeled by a friction coefficient $\gamma>0$ along with a random inflow of energy carried by the Brownian component.

We are interested here in sharp \textit{two-sided estimates} for the transition density of the process~\eqref{eq:Langevin_intro} absorbed when leaving the domain $D:=(0,1)\times\mathbb{R}$ \textcolor{black}{as well as an accurate description of its boundary regularity}. Let $$\tau_\partial:=\inf\{t>0:q_t\notin (0,1)\}.$$ 
We define the transition density of the process~\eqref{eq:Langevin_intro} absorbed outside of $D$ as the function $\mathrm{p}^D$ defined in $\mathbb{R}_+^*\times D\times D$ such that for all $t>0$, for all $(q,p)\in D$, for all Borel set $A\subset D$,
$$\mathbb{P}_{(q,p)}((q_t,p_t)\in A,\tau_\partial>t)=\int_A \mathrm{p}_t^D(q,p,q',p') \mathrm{d}q'\mathrm{d}p'.$$
We shall prove here the existence of time-dependent constants $c_t,c'_t>0$ (which may also depend on the coefficients of~\eqref{eq:Langevin_intro})  and explicit functions $H,H^*$ defined on $D$ such that for all $t>0$,
\begin{equation}\label{intro:two-sided}
    \forall (q,p),(q',p')\in D,\qquad c'_tH(q,p)H^*(q',p')\leq\mathrm{p}^D_t(q,p,q',p')\leq c_tH(q,p)H^*(q',p').
\end{equation}
These estimates have been proved in the literature for multidimensional elliptic processes such as
\begin{equation}\label{eq:elliptic diff intro}
    \mathrm{d}\overline{q}_t=F(\overline{q}_t)\mathrm{d}t+\sigma(\overline{q}_t)\mathrm{d}B_t,
\end{equation}
on a smooth bounded domain $\mathcal{O}$ when $F,\sigma$ are smooth and $\sigma$ is uniformly elliptic, see for instance~\cite{GQZ,KnobPart,kim2006two}. The proofs often rely on sharp estimates on the associated Green function which have been obtained for uniformly elliptic diffusions in~\cite{gruter1982green}. In our case, given the absence of noise in the $q$-direction in~\eqref{eq:Langevin_intro}, the associated infinitesimal generator of~\eqref{eq:Langevin_intro} is not elliptic anywhere in the domain and therefore such methods do not apply in this case.  

Sharp two-sided estimates are important as they can help us understand how the absorbed semigroup of~\eqref{eq:Langevin_intro} and associated eigenvectors decay at the boundary $\partial D$, \textcolor{black}{highlighting in particular the boundary regularity of such functions}.  In addition to that, these estimates allow us to obtain the sharp long-time behaviour, with respect to the initial condition, of the process conditioned on not being absorbed. This property is called \textit{uniform conditional ergodicity} and was obtained for elliptic diffusions such as~\eqref{eq:elliptic diff intro} in~\cite{KnobPart,champagnat2018criteria} ensuring the existence of a constant $C>0$ such that for all $t\geq0$, for any probability distribution $\theta$ on $\mathcal{O}$,
\begin{equation}\label{eq:cv unif vers qsd intro}
    \left\Vert\mathbb{P}_\theta\left(\overline{q}_t \in \cdot | \overline{\tau}_\partial > t\right) - \mu(\cdot)\right\Vert_{TV}\leq C \mathrm{e}^{-\alpha t},
\end{equation}
where $\|\cdot\|_{TV}$ denotes the total-variation norm on the space of bounded signed measures on $\mathbb{R}^{2}$, $\overline{\tau}_\partial=\inf\{t>0:\overline{q}_t\notin\mathcal{O}\}$ and  $\mu$ is called the quasi-stationary distribution on $\mathcal{O}$ of the process~\eqref{eq:elliptic diff intro} (see for instance~\cite{Collet,VilMel,CattColMelMart} for a literature on quasi-stationary distributions). 

\textcolor{black}{After obtaining the two-sided estimates~\eqref{intro:two-sided} in this work, we shall prove the $C^\infty$ regularity of $\mathrm{p}^D_t$ up to the boundary excluding the boundary grazing points. In addition, we show that all the derivatives of $\mathrm{p}^D_t$ vanish in certain subsets of the boundary due to the kinetic nature of the process. Thanks to the two-sided estimates, we also show that the function $\mathrm{p}^D_t$ is exactly $(1/6,1/2)$ H\"older continuous in the $(q,p)$ coordinates in a neighborhood of grazing boundary points.} 

We are also able to extend the uniform convergence in~\eqref{eq:cv unif vers qsd intro} to the underdamped Langevin process~\eqref{eq:Langevin_intro} in the domain $D$. Let us mention that convergence results of the conditioned distribution towards the quasi-stationary distribution were already obtained recently for the underdamped Langevin process in~\cite{QSD1,GuiNectoux}. However, in these results the prefactor $C=C(\theta)$ strongly depends on the initial distribution $\theta$ and blows up at the boundary $\partial D$. 

Let us mention that this is the first work to provide two-sided estimates for the absorbed transition density, \textcolor{black}{establish its exact boundary regularity} and derive a uniform convergence result toward the quasi-stationary distribution for the underdamped Langevin process absorbed outside of a domain. 

We shall also provide  here explicit expressions of the functions $H$, $H^*$ involved in the two-sided estimates~\eqref{intro:two-sided}, thus showing explicitly how the first exit time probability behaves close to the boundary of the domain. As said previously, the main difficulty of such proofs is that the tools used in the literature for the study of two-sided estimates of diffusion processes require an ellipticity condition on the generator which is not valid anywhere here. From the results we obtain we see that there is in fact a difference of nature between the boundary behaviour of an elliptic diffusion process and a kinetic process like~\eqref{eq:Langevin_intro}. The approach used here shall be different and will rely on a careful study of the first exit time probability. In the elliptic case, its boundary behaviour shall depend on the distance to the boundary of the domain. In the case of~\eqref{eq:Langevin_intro}, given the kinetic nature of the process the boundary behaviour is in fact expected to be more complex and shall also depend on the velocity and its sign. Examples of different behaviours are for instance provided in the next paragraph. 

If we integrate the two-sided estimates~\eqref{intro:two-sided} over $(q',p')\in D$, we obtain constants $c_t,c'_t>0$ such that for all $t>0$,
\begin{equation}\label{eq:proba estimates intro}
    \forall (q,p)\in D,\qquad c'_tH(q,p)\leq\mathbb{P}_{(q,p)}(\tau_\partial>t)\leq c_tH(q,p).
\end{equation}
Let us now take a look at the behaviour of the first exit time probability $\mathbb{P}_{(q,p)}(\tau_\partial>t)$ at a fixed time $t>0$, depending on the vector $(q,p)\in D$, in the following two-dimensional example: 
\begin{figure}[h]  
\begin{center}
\includegraphics[scale=0.4]{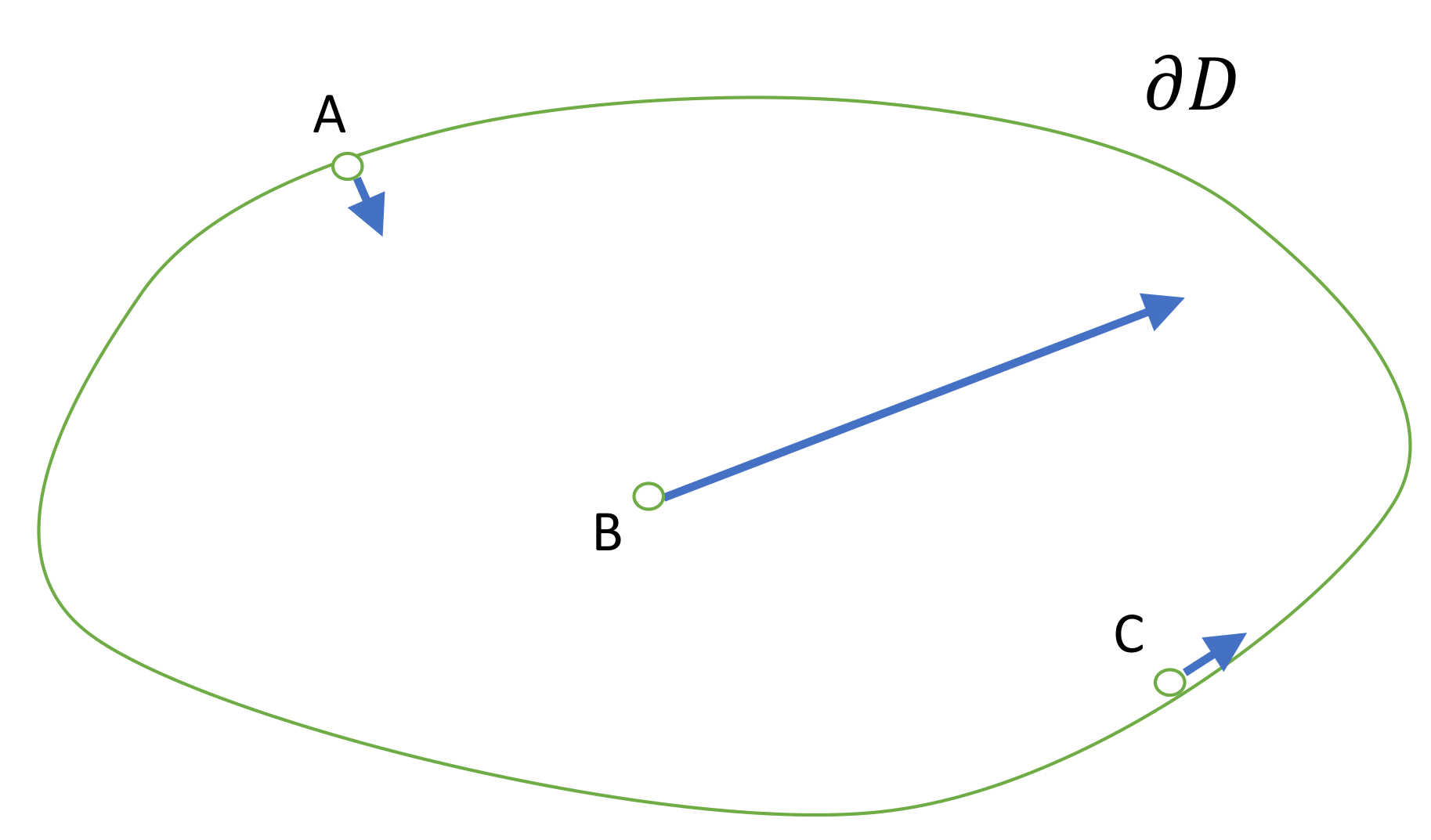}
\end{center}
\caption{Exit event from $D$.}
\label{Fig:exit event}
\end{figure}

  Consider the point $A\in\partial D$ with velocity directed towards the interior of $D$ in Figure~\ref{Fig:exit event}. In this case, even though $A\in\partial D$, the probability $\mathbb{P}_{(q,p)}(\tau_\partial>t)$ is positive because the process~\eqref{eq:Langevin_intro} re-enters the domain immediately almost-surely. On the contrary, it shall vanish if the velocity is directed towards the exterior of the domain $D$. Consider now the point $B\in D$ which is away from the boundary $\partial D$ with however a very large velocity. In this case the probability $\mathbb{P}_{(q,p)}(\tau_\partial>t)$ should be very small and is expected to vanish when the velocity goes to infinity, even though the distance of $B$ to the boundary is bounded from below. In fact, since the velocity is very high the process is likely to exit $D$ "quickly". However, this velocity should be high compared to the distance to the boundary in some sense which shall be clarified throughout the computations. Consider now the last case, the point $C\in D$ is close to the boundary with tangential velocity. In this case the velocity component does not play any role but given the random oscillations in the position coordinates, which are propagated by the velocity coordinates, the probability $\mathbb{P}_{(q,p)}(\tau_\partial>t)$ should be small as well but in this case its boundary behaviour shall mostly depend on the distance to the boundary. Therefore, the kinetic nature of the process exhibits multiple boundary behaviours depending on the distance to the boundary as well as the velocity.

The main objective of this work is to prove the two-sided estimates~\eqref{intro:two-sided} for the process~\eqref{eq:Langevin_intro} \textcolor{black}{when $F\in C^\infty(\mathbb{R})$, $\gamma\in\mathbb{R}$ and $\sigma>0$}. The strategy of this work is to first prove estimates on the first exit time probability as in~\eqref{eq:proba estimates intro} and then show that these estimates actually ensure the aimed two-sided estimates~\eqref{intro:two-sided} using estimates obtained on the transition density $\mathrm{p}_t^D$ in previous works~\cite{kFP,QSD1}. 

Nonetheless, let us mention that some recent works have focused on the analytic counterpart of~\eqref{eq:Langevin_intro}. Namely in~\cite{Vel}, the authors have shown that weak solutions to the parabolic equation $\partial_tu=\Delta_pu$ vanishing at the exiting boundary exhibit a Hölderian behaviour $(\alpha,3\alpha)$ in (position,velocity) on the grazing boundary points, for any $\alpha<1/6$. Actually, it shall follow from the two-sided estimates~\eqref{intro:two-sided} and with the expression of $H,H^*$ that the Hölderian behaviour on the grazing set is actually attained for $\alpha=1/6$ for the transition density $\mathrm{p}_t^D$.  
 
Some attention has also been drawn in the literature towards the long-time behavior of the following process when absorbed outside of the position half-line $(0,\infty)$:
\begin{equation}\label{eq:int_brownian_intro}
  \left\{
    \begin{aligned}
        &\mathrm{d}\widehat{q}_t=\widehat{p}_t \mathrm{d}t , \\
        &\mathrm{d}\widehat{p}_t=\mathrm{d}B_t.
    \end{aligned}
\right.  
\end{equation} 
The solution $(\widehat{q}_t)_{t\geq0}$ of~\eqref{eq:int_brownian_intro}, corresponds to a time integrated Brownian process, up to the initial conditions $(\widehat{q}_0,\widehat{p}_0)$:
\begin{equation}\label{integrated brownian}
    \forall t\geq0,\qquad  \widehat{q}_t=\widehat{q}_0+t\widehat{p}_0+\int_0^tB_s\mathrm{d}s.
\end{equation}
Namely, previous works in~\cite{McKean,Goldman,Gorkov,Lefebvre,lachal1} have led to an  explicit expression of the law of $(\widehat{\tau}_0, B_{\widehat{\tau}_0})$ starting from any couple $(\widehat{q}_0,\widehat{p}_0)\in\mathbb{R}_+^*\times\mathbb{R}$, where $\widehat{\tau}_0=\inf\{t>0:\widehat{q}_t=0\}$. 

The long-time behaviour of the probability $\mathbb{P}_{(q,p)}(\widehat{\tau}_0>t)$ was also studied in~\cite{sinai} where the author showed that this probability behaved as $h(q,p)/t^{1/4}$ for $q>0,p\in\mathbb{R}$ when $t\rightarrow\infty$. Later, an explicit expression of the prefactor $h(q,p)$ was provided in the following works~\cite{isozaki1994,GP}. This function $h$ is crucial in this work as it will appear in the definition of the explicit two-sided estimates~\eqref{intro:two-sided}.  

The literature regarding the study of the first exit event is less extensive for bounded position-domains. For instance, there is no explicit description of the law of the first exit event $(\tau_\partial, B_{\tau_\partial})$ from $D$. However, the long-time behavior of the first exit-time probability was studied in~\cite[Theorem 2.17]{QSD1} where it is shown that $\mathbb{P}_{(q,p)}(\tau_\partial>t)\underset{t\rightarrow\infty}{\sim}\phi(q,p)\mathrm{e}^{-\lambda_0 t}$ for some $\lambda_0>0$. The function $\phi$ is the principal eigenvector, up to a multiplicative constant, of the infinitesimal generator of~\eqref{eq:Langevin_intro}.  However no estimates on $\phi$ at the boundary $\partial D$ was obtained but shall follow in this work for~\eqref{eq:Langevin_intro} in the one-dimensional case, using~\eqref{intro:two-sided}.
  
This work is divided as follows: we shall first prove sharp estimates for the integrated Brownian process of the probability $\mathbb{P}_{(q,p)}(\widehat{\tau}_\partial>t)$ where $\widehat{\tau}_\partial=\inf\{t>0:\widehat{q}_t\notin(0,1)\}$. These estimates are then used to deduce two-sided estimates on its absorbed transition density. Then, using a Girsanov argument we will be able to extend these two-sided estimates to the process~\eqref{eq:Langevin_intro}. \textcolor{black}{Later, we also establish the exact boundary regularity of the absorbed transition density.} Finally, combining the long-time asymptotics of the absorbed semigroup obtained in~\cite{QSD1} and the two-sided estimates~\eqref{intro:two-sided} we shall obtain the uniform conditional ergodicity property stated in~\eqref{eq:cv unif vers qsd intro}.

We provide here a few notations used throughout this work before detailing in the next section the main results of this work.\medskip

\textbf{Notations:} For any subset $A$ of $\mathbb{R}^n$ ($n\geq1)$, we denote by:
\begin{enumerate}
    \item $\mathcal{B}(A)$ the Borel $\sigma$-algebra on $A$,
    \item for $1 \le p\le \infty$, $\mathrm{L}^p(A)$ the set of
      $\mathrm{L}^p$ scalar-valued functions on $A$ and $\Vert\cdot\Vert_{p}$ the associated norm, 
    \item $\mathcal{C}(A)$ (resp. $\mathcal{C}^b(A)$) the set of
      scalar-valued continuous (resp. continuous and bounded) 
      functions on $A$,
    \item for $1 \le k\le \infty$, $\mathcal{C}^k(A)$ 
      the set of scalar-valued $\mathcal{C}^k$  functions on $A$.
     \item for $a,b\in \mathbb{R}$, $a \wedge b = \min(a,b)$ and $a\lor b=\max(a,b)$
\end{enumerate}
Finally for families of functions $(f_t)_{t>0}$, $(g_t)_{t>0}$ defined on $D\subset\mathbb{R}^n$, we shall say that for all $t>0$, $f_t(x)\asymp g_t(x)$ if for all $t>0$, there exist constants $c_t,c'_t>0$ such that for all $x\in D$, 
$$c'_tg_t(x)\leq f_t(x)\leq c_tg_t(x).$$

\section{Main results}\label{sec:main results}
\subsection{Model}
Let $(\Omega,\mathcal{F},(\mathcal{F}_t)_{t\geq0},\mathbb{P})$ be the probability space under consideration. Consider the underdamped Langevin process $(q_t,p_t)_{t\geq0}$ solution to the following equation:
\textcolor{black}{\begin{equation}\label{eq:Langevin}
  \left\{
    \begin{aligned}
        &\mathrm{d}q_t=p_t \mathrm{d}t , \\
        &\mathrm{d}p_t=F(q_t)\mathrm{d}t-\gamma p_t\mathrm{d}t+\sigma\mathrm{d}B_t,
    \end{aligned}
\right.  
\end{equation}} 
where \textcolor{black}{$F\in C^\infty(\mathbb{R})$ with bounded successive derivatives}, $\gamma\in\mathbb{R}$, $\sigma>0$ and $(B_t)_{t\geq0}$ is a one-dimensional $(\mathcal{F}_t)_{t\geq0}$-Brownian process. Its infinitesimal generator, also called kinetic Fokker-Planck operator is defined for $(q,p)\in\mathbb{R}\times\mathbb{R}$ by:
\textcolor{black}{\begin{equation}\label{generateur Langevin qsd}
    \mathcal{L} =  p\partial_q+F(q)\partial_p -\gamma  p\partial_p+ \frac{\sigma^2}{2}\Delta_p, 
\end{equation}}
with formal adjoint $\mathcal{L}^*$ in $\mathrm{L}^2(\mathrm{d}q\mathrm{d}p)$ given by:
\textcolor{black}{\begin{equation}\label{generateur adjoint qsd}
    \mathcal{L}^*=-p\partial_q-F(q)\partial_p+\gamma \partial_p(p  \cdot )+\frac{\sigma^2}{2}\Delta_p. 
\end{equation}}

\noindent Let $\mathcal{O}:=(0,1)$ and $D:=\mathcal{O}\times\mathbb{R}$.
Let us introduce the following partition of $\partial D$: 
$$ \Gamma^+:=\{ (0,p): p<0 \}\cup\{ (1,p): p>0 \},$$
$$ \Gamma^-:=\{ (0,p): p>0 \}\cup\{ (1,p): p<0 \},$$
$$\Gamma^0:=\{ (0,0),(1,0) \}.$$
The sets $\Gamma^\pm$ correspond to the boundary points with exiting or entering velocity in $D$ and $\Gamma^0$ is the boundary set with zero velocity also called grazing set. Let $$\tau_\partial:=\inf\{t > 0: q_t\notin \mathcal{O}\}.$$
 
It was shown in~\cite[Theorem 2.20]{kFP} that the transition kernel of the Langevin process~\eqref{eq:Langevin} absorbed outside of $D$ admits a smooth transition density, i.e. there exists a function 
$$(t,(q,p),(q',p'))\in\mathbb{R}_+^*\times D\times D\mapsto \mathrm{p}_t^D(q,p,q',p')
    \in\mathcal{C}^\infty(\mathbb{R}_+^{*}\times D\times
    D)\cap\mathcal{C}(\mathbb{R}_+^*\times\overline{D}\times\overline{D})$$
such that for all $t>0$, $(q,p)\in \overline{D}$ and
    $A\in\mathcal{B}(D)$, 
\begin{equation}\label{transition kernel absorbed}
    \mathbb{P}_{(q,p)}((q_t,p_t)\in A,\tau_\partial>t)=\int_A \mathrm{p}_t^D(q,p,q',p') \mathrm{d}q'\mathrm{d}p',
\end{equation}
where we denote by $\mathbb{P}_{(q,p)}$ the probability measure under which $(q_0,p_0)=(q,p)$ almost surely.
\subsection{Two-sided estimates and boundary regularity}
The main result of this work is the two-sided estimates \textcolor{black}{and the boundary regularity} satisfied by the transition density $\mathrm{p}^D_t$. \textcolor{black}{In order to state this result, we define the following functions.}

\begin{align*}
g:z\in\mathbb{R}&\mapsto 
    \begin{cases}
        \left(\frac{2}{9}\right)^{1/6}zU\left(\frac{1}{6},\frac{4}{3},\frac{2}{9}z^3\right)&z>0,\\
        -\left(\frac{2}{9}\right)^{1/6}\frac{1}{6}zV\left(\frac{1}{6},\frac{4}{3},\frac{2}{9}z^3\right)&z<0,\\ 
        \left(\frac{2}{9}\right)^{-1/6}\Gamma(1/3)/\Gamma(1/6),&z=0,
\end{cases}
\end{align*}
where $U,V$ are the confluent hypergeometric functions, see~\cite[p. 256]{Olver}. The function $g$ also satisfies the following asymptotics see~\cite[Lemma 2.1, Sections 3 and 5]{GP}:
\begin{remark}[Asymptotics of $g$]\label{asymptotics g} 
The function $g$ is a positive, analytic, non-decreasing function satisfying
$$g(z)\sim 
    \begin{cases}
        \sqrt{z} &z\rightarrow+\infty,\\
        \left(\frac{2}{9}\right)^{-5/6}\frac{1}{6}\mathrm{e}^{-2\vert z\vert^3/9}\vert z\vert^{-5/2} &z\rightarrow-\infty. 
\end{cases}$$
\end{remark}
\noindent Let us define
\begin{equation}\label{def H with g}
    h:(q,p)\in D\mapsto q^{1/6}g(p/q^{1/3}).
\end{equation}
The function $h$ is used in~\cite{isozaki1994} and~\cite[Lemma 2.1]{GP} to characterize the long-time behaviour of \mbox{$\mathbb{P}_{(q,p)}(\widehat{\tau}_0>t)$} for the integrated Brownian process~\eqref{eq:int_brownian_intro} through the asymptotics:
\begin{equation}\label{long time asymptotics}
    \forall\;(q,p)\in \mathbb{R}_+^*\times\mathbb{R},\qquad\mathbb{P}_{(q,p)}(\widehat{\tau}_0>t)\underset{t\rightarrow\infty}{\sim}\frac{3\Gamma(1/4)}{2^{3/4}\pi^{3/2}}\frac{h(q,p)}{t^{1/4}}.
\end{equation}
Symmetrically, the function $h(1-q,-p)$ can be used to describe the long-time asymptotics of \mbox{$\mathbb{P}_{(q,p)}(\widehat{\tau}_1>t)$} where $\widehat{\tau}_1=\inf\{t>0:\widehat{q}_t=1\}$. This allows us to think that the function $H$ defined as follows:
\begin{equation}\label{def H}
    \forall (q,p)\in D,\qquad H(q,p)=h(q,p)\wedge h(1-q,-p) 
\end{equation}
captures the behaviour in $(0,1)$ of the probability $\mathbb{P}_{(q,p)}(\widehat{\tau}_\partial>t)$ for a fixed $t>0$. We shall in fact prove this result and later extend it to the process~\eqref{eq:Langevin}, using in particular a Girsanov argument, obtaining the following function: 
\textcolor{black}{\begin{equation}\label{eq:def H_alpha,beta,gamma,sigma}
    H_{F,\gamma,\sigma}(q,p)=H(q,p/\sigma^{2/3})\mathrm{e}^{(A^+(q)p^++A^-(q)p^-)/\sigma^2}\mathrm{e}^{\gamma p^2/2\sigma^2},
\end{equation}
where $p^\pm=(\pm p)\lor 0$, and
\begin{align}
A^-:q\in(0,1)&\mapsto-\frac{\sigma^2}{q^{2/3}}\int_0^q(F'(x)+\gamma^2/2)x^{2/3}\mathrm{d}x+F(q)\label{eq:def A-}\\
A^+:q\in(0,1)&\mapsto-\frac{\sigma^2}{(1-q)^{2/3}}\int_0^{1-q}(F'(1-x)+\gamma^2/2)x^{2/3}\mathrm{d}x-F(1-q)\label{eq:def A+}.
\end{align}}

\begin{remark}[Integrability of \textcolor{black}{$H_{F,\gamma,\sigma}$}]\label{rmk: prop H} There exists a constant $c>0$ such that for all $(q,p)\in D$,
\begin{align*}
    H_{F,\gamma,\sigma}(q,p)&\leq c\mathrm{e}^{cp^2}\big(\mathbb{1}_{p\geq0}h(1-q,-p/\sigma^{2/3})+\mathbb{1}_{p<0}h(q,p/\sigma^{2/3})\big)\\
    &=c\mathrm{e}^{cp^2}\big(\mathbb{1}_{p\geq0}(1-q)^{1/6}g(-p/\sigma^{2/3})+\mathbb{1}_{p<0}q^{1/6}g(p/\sigma^{2/3})\big)\\
    &\leq c\mathrm{e}^{cp^2} g(-\vert p\vert/\sigma^{2/3}). 
\end{align*}
Therefore, the asymptotics of $g$ in Remark~\ref{asymptotics g} ensure that $\textcolor{black}{H_{F,\gamma,\sigma}}\in\mathrm{L}^r(D)$ for any $r\in(0,\infty]$.
\end{remark} 

\textcolor{black}{Let us now state the main theorem of this work.}

\textcolor{black}{\begin{theorem}[Absorbed transition density]\label{thm:two-sided p_t^D}
The function $\mathrm{p}^D$ belongs to the space function
$$C^\infty(\mathbb{R}_+^*\times(\overline{D}\setminus\Gamma^0)\times(\overline{D}\setminus\Gamma^0)).$$
Moreover, for all $t>0$, 
\begin{equation}\label{two-sided p_t^D}
    \mathrm{p}_t^D(q,p,q',p')\asymp H_{F,\gamma,\sigma}(q,p)H_{F,-\gamma,\sigma}(q',-p').
\end{equation}
In addition, for any $m,n\geq0$, for all $t>0$,
\begin{align*}
\forall\;y\in\overline{D}\setminus\Gamma^0,\quad\forall\;(q,p)\in\Gamma^+,&\qquad\partial^{m}_{q}\partial^{n}_p\mathrm{p}^D_t((q,p),y)=0,\\
\forall\;x\in\overline{D}\setminus\Gamma^0,\quad\forall\;(q,p)\in\Gamma^-,&\qquad\partial^{m}_{q}\partial^{n}_p\mathrm{p}^D_t(x,(q,p))=0.
\end{align*}
\end{theorem}}

\begin{remark}[\textcolor{black}{H\"older continuity at $\Gamma^0$}]\label{rk:Holder reg} Given the asymptotics of $g$ when $p\rightarrow+\infty$ in Remark~\ref{asymptotics g}, there exists a constant $C>0$ such that   
\begin{align*}
    \forall\;q>0,\,\forall\;p\in\mathbb{R},\qquad h(q,p)&\leq \mathbb{1}_{p\leq0}q^{1/6}g(0)+\mathbb{1}_{p\geq0}C(q^{1/6}+\sqrt{p})\\
    &\leq C'(q^{1/6}+\sqrt{p_+}),
\end{align*}
for some constant $C'>0$ where $p_+=p\lor0$. \textcolor{black}{As a result, the transition density $\mathrm{p}^D_t$ is $(1/6,1/2)$ H\"older continuous in the $(q,p)$ coordinates at such points, hence the optimality of the boundary regularity stated in Theorem~\ref{thm:two-sided p_t^D}.}

This result can be related to the work by Hwang, Jang and Velazquez in~\cite{Vel} where the authors showed that when $F\equiv0$, $\gamma=0$ weak solutions to $\partial_tu=\mathcal{L}u$ with zero boundary condition on $\Gamma^+$ are $(\alpha,3\alpha)$-Hölderian at the boundary $\Gamma^0$ for any $\alpha\in(0,1/6)$. Here, this Hölder regularity is shown to actually hold for $\alpha=1/6$.
\end{remark}
 
 The two-sided estimates above can be propagated to the principal eigenvectors $\phi$ (resp. $\psi$) of the generator $\mathcal{L}$ (resp. $\mathcal{L}^*$) obtained in~\cite[Theorem 2.13]{QSD1} and recalled hereafter. Let $(\lambda_0,\phi)$ be the unique solution, up to a multiplicative constant on $\phi$, such that  $\phi\in\mathcal{C}^{2}(D)\cap\mathcal{C}^b(D\cup\Gamma^+)$ is a non-zero, non-negative classical solution to the following problem 
\begin{equation}\label{vp droite}
  \left\{
    \begin{aligned}
        \mathcal{L}\phi(q,p) &=-\lambda_0 \phi(q,p)&& \quad (q,p)\in D,\\
        \phi(q,p) &=0 && \quad (q,p)\in \Gamma^+.
    \end{aligned}
\right. 
\end{equation}   
Let also $(\lambda_0,\psi)$ be the unique solution, up to a multiplicative constant on $\psi$, such that  $\psi\in\mathcal{C}^{2}(D)\cap\mathcal{C}^b(D\cup\Gamma^-)$ is a non-zero, non-negative classical solution to:
\begin{equation}\label{vp gauche} 
  \left\{
    \begin{aligned}
        \mathcal{L}^*\psi(q,p) &=-\lambda_0 \psi(q,p)&& \quad (q,p)\in D,\\
        \psi(q,p) &=0 && \quad (q,p)\in \Gamma^-.
    \end{aligned}
\right. 
\end{equation}
In particular, if we choose the specific $\psi$ which satisfies $\int_D\psi=1$ we can define the following probability measure $\mu$ on $D$:
\begin{equation}\label{eq: def qsd mu}
    \mu(\mathrm{d}q\mathrm{d}p)=\psi(q,p)\mathrm{d}q\mathrm{d}p
\end{equation}
which is called the quasi-stationary distribution on $D$ of the process~\eqref{eq:Langevin}~\cite[Theorem 2.14]{QSD1}. \textcolor{black}{Note that the following corollary is an immediate consequence of~\eqref{two-sided p_t^D} using~\cite[Theorem 2.13]{QSD1}.}

\begin{corollary}[Eigenvectors estimates]\label{coroll phi psi}
$$\phi(q,p)\asymp H_{F,\gamma,\sigma}(q,p),\qquad\psi(q,p)\asymp H_{F,-\gamma,\sigma}(q,-p).$$  
\end{corollary}
\begin{remark}[Comparability]
In contrast with the elliptic case, see for instance~\cite[Proposition 3]{GQZ}, we do not have that $\phi(q,p)\asymp\psi(q,p)$. However, one has that $\phi(q,-p)\asymp\psi(q,p)$ if $\gamma=0$.
\end{remark}
Furthermore, Theorem~\ref{thm:two-sided p_t^D} allows us to define the semigroup associated to the absorbed transition kernel~\eqref{transition kernel absorbed} on the Banach space $\mathrm{L}_{\textcolor{black}{H_{F,\gamma,\sigma}}}$ given by
$$\mathrm{L}_{\textcolor{black}{H_{F,\gamma,\sigma}}}:=\big\{f\text{ measurable }: (q',p')\in D\mapsto \vert f(q',p')\vert \textcolor{black}{H_{F,-\gamma,\sigma}(q',-p')}\in\mathrm{L}^1(D)\big\},$$
endowed with the norm $\Vert f\Vert_{\textcolor{black}{H_{F,\gamma,\sigma}}}:=\int_D\vert f(q',p')\vert \textcolor{black}{H_{F,-\gamma,\sigma}(q',-p')}\mathrm{d}q'\mathrm{d}p'$.

\begin{remark}[Space function $\textcolor{black}{\mathrm{L}_{\textcolor{black}{H_{F,\gamma,\sigma}}}}$]\label{rk:integrability against H} The inequality on $\textcolor{black}{H_{F,\gamma,\sigma}}$ provided in Remark~\ref{rmk: prop H} combined with the asymptotics of $g$ in Remark~\ref{asymptotics g} ensures that $\mathrm{L}_{\textcolor{black}{H_{F,\gamma,\sigma}}}$ contains the set of functions:
$$(q,p)\in D\mapsto\eta(q)\mathrm{e}^{c\vert p\vert^3/\sigma^2},$$
where $\eta\in\mathrm{L}^1(\mathcal{O})$ and $c\in[0,2/9)$.
\end{remark}

The semigroup $(P^D_t)_{t\geq0}$ associated to the transition kernel~\eqref{transition kernel absorbed} in $\mathrm{L}_{\textcolor{black}{H_{F,\gamma,\sigma}}}$ is defined as follows: $P^D_0f=f$ for $f\in\mathrm{L}_{\textcolor{black}{H_{F,\gamma,\sigma}}}$ and for $t>0$, 
\begin{equation}\label{eq:def semigroup}
    \forall f\in \mathrm{L}_{\textcolor{black}{H_{F,\gamma,\sigma}}}, \forall (q,p)\in D,\qquad P^D_tf(q,p) =\mathbb{E}_{(q,p)}\left[f(q_t,p_t)\mathbb{1}_{\tau_\partial>t}\right]=\int_D f(q',p')\mathrm{p}^D_t(q,p,q',p')\mathrm{d}q'\mathrm{d}p'.
\end{equation}
Theorem~\ref{thm:two-sided p_t^D} and Corollary~\ref{coroll phi psi} ensure the following immediate  corollary.
\begin{corollary}[Absorbed semigroup estimates]\label{corollary} 
Let $f$ be a non-negative function in $\mathrm{L}_{\textcolor{black}{H_{F,\gamma,\sigma}}}$, then for all $t>0$, there exist $c_t>0$, $c'_t>0$ such that for all $f\in\mathrm{L}_{\textcolor{black}{H_{F,\gamma,\sigma}}}$,
$$\forall (q,p)\in D,\qquad c'_t\Vert f\Vert_{\textcolor{black}{H_{F,\gamma,\sigma}}}\textcolor{black}{H_{F,\gamma,\sigma}}(q,p)\leq P^D_tf(q,p)\leq c_t\Vert f\Vert_{\textcolor{black}{H_{F,\gamma,\sigma}}}\textcolor{black}{H_{F,\gamma,\sigma}}(q,p).$$
\textcolor{black}{In particular, the absorbed semigroup also satisfies the decay estimates in Remark~\ref{rmk: prop H} and the H\"older regularity detailed in Remark~\ref{rk:Holder reg}.}
\end{corollary} 

The strategy of the proof of Theorem~\ref{thm:two-sided p_t^D}  consists in showing first sharp estimates on the first exit time probability as in~\eqref{eq:proba estimates intro}. Second, these estimates are shown to ensure the two-sided estimates using the proposition stated below. Let us emphasize that the equivalence between two-sided estimates and sharp estimates on the first exit time probability remains also valid in large dimension. Therefore, we shall state the proposition below for general multidimensional Langevin processes.
\begin{proposition}[Equivalence of two-sided estimates]\label{prop:first exit estimates implies two sided}
Let $d\geq1$. Let $\mathcal{O}$ be a $\mathcal{C}^2$ bounded connected open set of $\mathbb{R}^d$. Let $D=\mathcal{O}\times\mathbb{R}^d$ and let $(q^{F,\gamma,\sigma}_t,p^{F,\gamma,\sigma}_t)_{t\geq0}$ be the process in $\mathbb{R}^d\times\mathbb{R}^d$ solution to
\begin{equation}\label{eq:Langevin multi dimensional}
  \left\{
    \begin{aligned}
        &\mathrm{d}q^{F,\gamma,\sigma}_t=p^{F,\gamma,\sigma}_t \mathrm{d}t , \\
        &\mathrm{d}p^{F,\gamma,\sigma}_t=F(q^{F,\gamma,\sigma}_t)\mathrm{d}t-\gamma p^{F,\gamma,\sigma}_t\mathrm{d}t+\sigma\mathrm{d}B_t,
    \end{aligned}
\right.  
\end{equation}
where $F\in\mathcal{C}^\infty(\mathbb{R}^d)$, $\gamma\in\mathbb{R}$ and $\sigma>0$. Let $\tau^{F,\gamma,\sigma}_\partial=\inf\{t>0:q^{F,\gamma,\sigma}_t\notin\mathcal{O}\}.$ Assume that there exists a function $H_{F,\gamma,\sigma}$ in $D$ such that for all $t>0$, 
\begin{equation}\label{eq:hypothese prop 2.9}
\mathbb{P}_{(q,p)}(\tau^{F,\gamma,\sigma}_\partial>t)\asymp H_{F,\gamma,\sigma}(q,p)\quad\text{and}\quad\mathbb{P}_{(q,p)}(\tau^{F,-\gamma,\sigma}_\partial>t)\asymp H_{F,-\gamma,\sigma}(q,p).
\end{equation}
Then, for all $t>0$, 
$$\mathrm{p}_t^D(q,p,q',p')\asymp H_{F,\gamma,\sigma}(q,p)H_{F,-\gamma,\sigma}(q',-p').$$
\end{proposition}

As a result, in the general multidimensional Langevin case it is sufficient to obtain sharp estimates on the first exit time probability to obtain the two-sided estimates.
\subsection{Long-time asymptotics}
The long-time asymptotics obtained hereafter follow from the combination of the previous two-sided estimates with the long-time convergence of the absorbed semigroup described in~\cite{QSD1}. In particular, they allow us to sharpen a convergence result, stated in~\cite[Theorem 2.22]{QSD1} regarding the convergence of the law of the process~\eqref{eq:Langevin} conditioned on not being absorbed. Unlike the results found in the literature in~\cite{QSD1,GuiNectoux,benaim}, the convergence result below is stated with a prefactor independent of the initial distribution $\theta$. 
\begin{theorem}[Long-time asymptotics]\label{thm:long time}
There exists $\alpha>0$ such that for all $t_0>0$ there exists $C_{t_0}>0$ such that for all $t\geq t_0$ and for all $(q,p),(q',p')\in D$, 
\begin{equation}\label{eq:long-time cv constants}
\left\vert\mathrm{e}^{\lambda_0t}\mathrm{p}_t^D(q,p,q',p')-\frac{\phi(q,p)\psi(q',p')}{\int_D\phi\psi}\right\vert\leq C\phi(q,p)\psi(q',p')\mathrm{e}^{-\alpha t}.
\end{equation}
In addition, for all  $f\in\mathrm{L}_{\textcolor{black}{H_{F,\gamma,\sigma}}}$, for any probability measure $\theta$ on $D$,
\begin{equation}\label{semigroup conditionnel ineq}
    \left\vert\mathbb{E}_\theta\left[f(q_t,p_t) | \tau_\partial > t\right] - \int_Df\mathrm{d}\mu\right\vert\leq C_{t_0}\Vert f\Vert_{\textcolor{black}{H_{F,\gamma,\sigma}}} \mathrm{e}^{-\alpha t},
\end{equation}  
where $\mu$ is the quasi-stationary distribution defined in~\eqref{eq: def qsd mu}.
\end{theorem}

\begin{remark}[Total-variation convergence]  
In particular, taking the supremum in~\eqref{semigroup conditionnel ineq} for $f \in \mathrm{L}^\infty(D)\subset\mathrm{L}_{\textcolor{black}{H_{F,\gamma,\sigma}}}$ such that $\Vert f\Vert_{\mathrm{L}^\infty(D)} \leq 1$, one has the existence of $\alpha,C>0$ such that for all $t\geq0$, for all probability measure $\theta$ on $D$,
\begin{equation}\label{cv qsd tv}
    \left\Vert\mathbb{P}_\theta\left((q_t,p_t) \in \cdot | \tau_\partial > t\right) - \mu(\cdot)\right\Vert_{TV}\leq C \mathrm{e}^{-\alpha t}.
\end{equation}
 Following a criterion established in~\cite[Proposition 3.8]{V}, the convergence in~\cite{QSD1} can be extended directly to the uniform convergence~\eqref{cv qsd tv} if there exists $t_0>0$ and a compact set $K_0\subset D$ such that the following estimate is satisfied, 
\begin{equation}\label{eq:return compact}
    \inf_{(q,p)\in D}\frac{\mathbb{P}_{(q,p)}((q_{t_0},p_{t_0})\in K_0,\tau_\partial>t_0)}{\mathbb{P}_{(q,p)}(\tau_\partial>t_0)}>0,
\end{equation}
which follows directly from the two-sided estimates in Theorem~\ref{thm:two-sided p_t^D}. The purpose of the estimates in Theorem~\ref{thm:long time} is therefore mainly to extend this convergence to the largest set of functions $\mathrm{L}_{\textcolor{black}{H_{F,\gamma,\sigma}}}$.
\end{remark} 
This article is divided as follows: Section~\ref{sec:first exit time estimates} is devoted to the proof of sharp and explicit estimates on the first exit time probability of the Langevin process~\eqref{eq:Langevin}. Section~\ref{sec:two-sided density} utilizes these result to establish the two-sided estimates, \textcolor{black}{the boundary regularity of the absorbed transition density} and its long-time asymptotics. Namely, we prove Proposition~\ref{prop:first exit estimates implies two sided} in Section~\ref{sec:equivalence two sided} which yields the two-sided estimates from the results of Section~\ref{sec:first exit time estimates}. Corollary~\ref{coroll phi psi} is also proven in Section~\ref{sec:equivalence two sided}. The boundary regularity of the transition density $\mathrm{p}^D_t$ is proven in Section~\ref{prop:bdy regul density} which concludes the proof of Theorem~\ref{thm:two-sided p_t^D}. Finally, the long-time asymptotics described in Theorem~\ref{thm:long time} are proven in Section~\ref{sec:long time asymptotics}.
 
\section{First exit time probability estimates}\label{sec:first exit time estimates}

In this section we shall prove sharp and explicit estimates for the first exit time probability of the Langevin process~\eqref{eq:Langevin}.
\begin{proposition}[First exit time probability estimates]\label{prop:estimates first exit time Langevin}
\textcolor{black}{For all $t>0$,  
\begin{equation}\label{eq:asymptot proba langevin}
\mathbb{P}_{(q,p)}(\tau_\partial>t)\asymp H_{F,\gamma,\sigma}(q,p).
\end{equation}}
\end{proposition}

The proof is first completed in Section~\ref{sec:two-sided proba integrated brownian process} for the integrated Brownian process~\eqref{eq:int_brownian_intro} and extended in Section~\ref{sec:first exit time proba langevin} to the Langevin process~\eqref{eq:Langevin} using a Girsanov argument. 

\subsection{Integrated Brownian process}\label{sec:two-sided proba integrated brownian process}

Let us consider the following process for $\sigma>0$:
\begin{equation}\label{eq:int brownian motion with sigma}
  \left\{
    \begin{aligned}
        &\mathrm{d}\widehat{q}^\sigma_t=\widehat{p}^\sigma_t \mathrm{d}t , \\
        &\mathrm{d}\widehat{p}^\sigma_t=\sigma\mathrm{d}B_t.
    \end{aligned}
\right.  
\end{equation}
We denote by $\widehat{\tau}_\partial^\sigma$ its first exit time from $D$. The main result of this section is the following.
\begin{proposition}[First exit time probability estimates]\label{prop:two-sided proba final}
For all $t>0$,  
$$\mathbb{P}_{(q,p)}(\widehat{\tau}^\sigma_\partial>t)\asymp H(q,p/\sigma^{2/3}),$$
where $H$ is defined in~\eqref{def H}.
\end{proposition}
In addition, recall that for $c>0$, the integrated Brownian process satisfies the following equality-in-law:
\begin{equation}\label{eq:time scale change int brownian}
    \left(\int_0^{ct}B_s\mathrm{d}s\right)_{t\geq0}\overset{\mathcal{L}}{=}\left(c^{3/2}\int_0^{t}B_s\mathrm{d}s\right)_{t\geq0}.
\end{equation}
Therefore, for any $t>0$, $(q,p)\in D$,
\begin{equation}\label{eq:transformation sans sigma}
    \mathbb{P}_{(q,p)}(\widehat{\tau}^\sigma_\partial>t)=\mathbb{P}_{(q,p/\sigma^{2/3})}(\widehat{\tau}_\partial>\sigma^{2/3}t),
\end{equation}
where $\widehat{\tau}_\partial$ corresponds to the first exit time from $D$ of~\eqref{eq:int brownian motion with sigma} for $\sigma=1$. As a result, it is sufficient to complete the proof in the case $\sigma =1$. We shall from now on denote simply by $(\widehat{q}_t,\widehat{p}_t)_{t\geq0}$ the process~\eqref{eq:int brownian motion with sigma} defined with $\sigma=1$ and by $\widehat{\tau}_\partial$ its first exit time from $D$.

Therefore we shall prove the following proposition which directly implies Proposition~\ref{prop:two-sided proba final}.

\begin{proposition}[First exit time probability estimates]\label{prop:two-sided proba}
For all $t>0$,  
$$\mathbb{P}_{(q,p)}(\widehat{\tau}_\partial>t)\asymp H(q,p).$$
\end{proposition}

 The proof is divided as follows: Section~\ref{upper-bound subsection} is devoted to the proof of the upper-bound and Section~\ref{lower-bound subsection} focuses on the proof of the lower-bound. In order to achieve this proof we will use mostly the two following properties satisfied by the integrated Brownian process along with the long-time asymptotics~\eqref{long time asymptotics} from~\cite{isozaki1994,GP}.  

Let us define $\widehat{\tau}_0,\widehat{\tau}_1$ as the following hitting times $$\widehat{\tau}_0:=\inf\{t > 0: \widehat{q}_t=0\},\qquad\widehat{\tau}_1:=\inf\{t > 0: \widehat{q}_t=1\}.$$
We shall make use of the following properties.
\begin{remark}[Timescale change]\label{rmk scale change}
Given~\eqref{eq:time scale change int brownian}, it follows that for all $q>0$, $p\in \mathbb{R}$ and $\lambda>0$,
$$\mathbb{P}_{(q,p)}(\widehat{\tau}_0>t)=\mathbb{P}_{(\lambda^3q,\lambda p)}(\widehat{\tau}_0>\lambda^2t).$$
\end{remark}
\begin{remark}[Invariant transformation]\label{invariant proba}
Let us notice that the equality-in-law $\left(B_t\right)_{t\geq0}\overset{\mathcal{L}}{=}\left(-B_t\right)_{t\geq0}$ ensures that  
$$\forall (q,p)\in D,\qquad\mathbb{P}_{(q,p)}\left(\widehat{\tau}_\partial>t\right)=\mathbb{P}_{(1-q,-p)}\left(\widehat{\tau}_\partial>t\right),\text{ and}\quad\mathbb{P}_{(1-q,-p)}\left(\widehat{\tau}_0>t\right)=\mathbb{P}_{(q,p)}\left(\widehat{\tau}_1>t\right).$$
\end{remark} 
\subsubsection{Upper-bound on the first exit time probability}\label{upper-bound subsection}
The goal of this subsection is to provide an upper-bound on the probability $\mathbb{P}_{(q,p)}(\widehat{\tau}_\partial>t)$. Namely, we shall prove the following proposition.
 
\begin{proposition}[Upper-bound]\label{prop:upper-bound}
For all $t>0$ there exists $c_t>0$ such that
\begin{equation}\label{ineq:prop-upper-bound}
    \forall (q,p)\in D,\qquad \mathbb{P}_{(q,p)}(\widehat{\tau}_\partial>t)\leq c_t\,H(q,p).
\end{equation}    
\end{proposition}  
We resort in particular to the following lemma.
\begin{lemma}[Martingales]\label{lemma martingale}
For all $(q,p)\in D$, the processes $(h(\widehat{q}_{t\land\widehat{\tau}_\partial},\widehat{p}_{t\land\widehat{\tau}_\partial}))_{t\geq0}$ and $(h(1-\widehat{q}_{t\land\widehat{\tau}_\partial},-\widehat{p}_{t\land\widehat{\tau}_\partial}))_{t\geq0}$ are $(\mathcal{F}_t)_{t\geq0}$-martingales under $\mathbb{P}_{(q,p)}$, where $(\mathcal{F}_t)_{t\geq0}$ is the natural filtration of the Brownian process.
\end{lemma}
\begin{proof}
Let $(q,p)\in D$. Since $h\in\mathcal{C}^{\infty}(D)$ and $\mathcal{L}h=0$ on $D$, see~\cite[Lemma 2.1]{GP}, the Itô formula ensures that both processes are $(\mathcal{F}_t)_{t\geq0}$- local martingales under $\mathbb{P}_{(q,p)}$. They are martingales if one can prove that for all $t>0$, 
$$\mathbb{E}_{(q,p)}\left[\sup_{s\in[0,t\land\widehat{\tau}_\partial]} h(\widehat{q}_s,\widehat{p}_s)\right]<\infty \qquad\text{and}\qquad\mathbb{E}_{(q,p)}\left[\sup_{s\in[0,t\land\widehat{\tau}_\partial]}h(1-\widehat{q}_s,-\widehat{p}_s)\right]<\infty.$$
Using the definition of $h$ in~\eqref{def H with g} and the asymptotics of $g$ at $+\infty$ in Remark~\ref{asymptotics g}, one can see that both inequalities follow if one can show that
$$\mathbb{E}_{(q,p)}\left[\sup_{s\in[0,t]}\sqrt{\left\vert \widehat{p}_s\right\vert}\right]<\infty.$$
The Cauchy-Schwarz inequality ensures that it is sufficient to prove that
$$\mathbb{E}\left[\sup_{s\in[0,t]}\left\vert B_s\right\vert\right]<\infty.$$
Moreover,
\begin{align*}
    \mathbb{E}\left[\sup_{s\in[0,t]}\left\vert B_s\right\vert\right]&=\int_0^\infty\mathbb{P}\left(\sup_{s\in[0,t]}\left\vert B_s\right\vert>x\right)\mathrm{d}x\\ 
    &\leq\int_0^\infty4\mathbb{P}\left(B_t>x\right)\mathrm{d}x.
\end{align*}
Using the inequality $\mathbb{P}\left(B_t>x\right)\leq1\wedge\frac{\mathrm{e}^{-x^2/2t}}{\sqrt{2\pi}}\frac{\sqrt{t}}{x}$ concludes the proof.
\end{proof}
Let us now prove Proposition~\ref{prop:upper-bound}.
\begin{proof}[Proof of Proposition~\ref{prop:upper-bound}]
Let $t>0$, $(q,p)\in D$. Assume that $\frac{\vert p\vert}{q^{1/3}}\leq3/t$, one has using Remark~\ref{rmk scale change} that
\begin{align*}
    \mathbb{P}_{(q,p)}(\widehat{\tau}_\partial>t)&\leq\mathbb{P}_{(q,p)}(\widehat{\tau}_0>t)\\
    &=\mathbb{P}_{(1,p/q^{1/3})}(\widehat{\tau}_0>t/q^{2/3})\\
    &\leq \mathbb{P}_{(1,3/t)}(\widehat{\tau}_0>t/q^{2/3}),
\end{align*}
 since $\mathbb{P}_{(q,p)}(\widehat{\tau}_0>t)$ is a non-decreasing function of $p$. Therefore,
 \begin{align*}
    \frac{\mathbb{P}_{(q,p)}(\widehat{\tau}_\partial>t)}{h(q,p)}&\leq \frac{\mathbb{P}_{(1,3/t)}(\widehat{\tau}_0>t/q^{2/3})}{q^{1/6}g(p/q^{1/3})}\\ 
    &\leq \frac{1}{g(-3/t)}\frac{\mathbb{P}_{(1,3/t)}(\widehat{\tau}_0>t/q^{2/3})}{q^{1/6}},
\end{align*} 
since $g$ is a non-decreasing function. Besides, the term in the right-hand side of the inequality above is bounded when $q\rightarrow0$. In fact, using the long-time asymptotics~\eqref{long time asymptotics} from~\cite{isozaki1994,GP} one has the existence of a constant $\alpha_t>0$ such that
$$\frac{\mathbb{P}_{(1,3/t)}(\widehat{\tau}_0>t/q^{2/3})}{q^{1/6}}\underset{q\rightarrow0}{\sim}\frac{\alpha_tq^{1/6}}{q^{1/6}}=\alpha_t.$$ 
This ensures that for any $t>0$, there exists $\beta_t>0$ such that for all $(q,p)\in D$ satisfying $\vert p\vert/q^{1/3}\leq 3/t$,
\begin{equation}\label{ineq compact upper-bound}
    \mathbb{P}_{(q,p)}(\widehat{\tau}_\partial>t)\leq \beta_t h(q,p).
\end{equation}
Assume now that $\frac{\vert p\vert}{q^{1/3}}> 3/t$.
Let $$\widehat{\pi}_t:=\inf\{s\geq0: \vert \widehat{p}_s\vert/\widehat{q}_s^{1/3}\leq 3/t\}.$$
Necessarily, $\widehat{\tau}_\partial\wedge\widehat{\pi}_t\leq t/2$ almost surely. In fact if $\widehat{\tau}_\partial\wedge\widehat{\pi}_t>t/2$, then, depending on the sign of $\widehat{p}_0$, one has by continuity of the trajectory $(\widehat{p}_s)_{s\geq0}$ that, almost-surely, for all $s\in[0,t/2]$, 
$$\frac{\widehat{p}_s}{\widehat{q}^{1/3}_s}=\frac{\frac{\mathrm{d}\widehat{q}_s}{\mathrm{d}s}}{\widehat{q}^{1/3}_s} \leq -3/t\quad\text{or}\quad \frac{\widehat{p}_s}{\widehat{q}^{1/3}_s}=\frac{\frac{\mathrm{d}\widehat{q}_s}{\mathrm{d}s}}{\widehat{q}^{1/3}_s}\geq 3/t.$$  As a result, integrating $\widehat{p}_s/\widehat{q}^{1/3}_s$ over $s\in[0,t/2]$, one obtains that
$$\widehat{q}_{t/2}^{2/3}\leq q^{2/3}-1<0,\qquad\text{or}\qquad \widehat{q}_{t/2}^{2/3}\geq q^{2/3}+1>1,$$
which contradicts the fact that $\widehat{\tau}_\partial>t/2$. As a result, one has by the strong Markov property and~\eqref{ineq compact upper-bound} 
\begin{align*}
    \mathbb{P}_{(q,p)}(\widehat{\tau}_\partial>t)&=\mathbb{E}_{(q,p)}\left[\mathbb{1}_{\widehat{\tau}_\partial>\widehat{\pi}_t}\mathbb{P}_{(\widehat{q}_{\widehat{\pi}_t},\widehat{p}_{\widehat{\pi}_t})}\left(\widehat{\tau}_\partial>t-s\right)\vert_{s=\widehat{\pi}_t}\right]\\
    &\leq\mathbb{E}_{(q,p)}\left[\mathbb{1}_{\widehat{\tau}_\partial>\widehat{\pi}_t}\mathbb{P}_{(\widehat{q}_{\widehat{\pi}_t},\widehat{p}_{\widehat{\pi}_t})}\left(\widehat{\tau}_\partial>t/2\right)\right]\\
    &\leq \beta_{t/2}\,\mathbb{E}_{(q,p)}\left[\mathbb{1}_{\widehat{\tau}_\partial>\widehat{\pi}_t}h(\widehat{q}_{\widehat{\pi}_t},\widehat{p}_{\widehat{\pi}_t})\right]\\ 
    &\leq \beta_{t/2}\,\mathbb{E}_{(q,p)}\left[h(\widehat{q}_{\widehat{\pi}_t\wedge\widehat{\tau}_\partial},\widehat{p}_{\widehat{\pi}_t\wedge\widehat{\tau}_\partial})\right]=\beta_{t/2} h(q,p), 
\end{align*}
by Lemma~\ref{lemma martingale} and Doob's optional sampling theorem since $\widehat{\pi}_t\wedge\widehat{\tau}_\partial\leq t/2$ almost-surely. Therefore, for all $t>0$, there exists $c_t>0$ such that for all $(q,p)\in D$,
$$\mathbb{P}_{(q,p)}(\widehat{\tau}_\partial>t)\leq c_t h(q,p).$$
Using Remark~\ref{invariant proba}, we also have that 
$$\mathbb{P}_{(q,p)}(\widehat{\tau}_\partial>t)=\mathbb{P}_{(1-q,-p)}(\widehat{\tau}_\partial>t)\leq c_th(1-q,-p),$$ 
hence~\eqref{ineq:prop-upper-bound}, which concludes the proof.

\end{proof}
\subsubsection{Lower-bound on the first exit time probability}\label{lower-bound subsection}
We shall prove in this section the lower-bound in Proposition~\ref{prop:two-sided proba}. This proof is more complex as it requires a careful study depending on the location in the phase space.  We represent below the domain $D$ with its boundary \textcolor{black}{$\Gamma^+$}, \textcolor{red}{$\Gamma^-$} and \textcolor{green}{$\Gamma^0$} along with a partition \textbf{1}, \textbf{2} and \textbf{3} of $D$ used in the lower-bound proof.
\begin{center}
    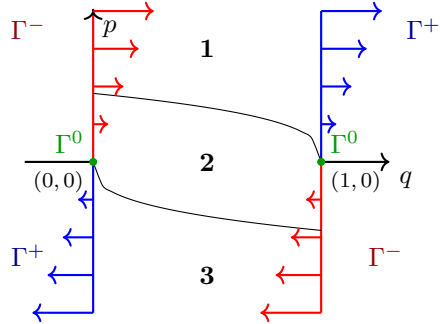
\begin{figure}[h]
        \centering
        \begin{tikzpicture}
        \input{plot_decomposition} 
        \end{tikzpicture} 
        \caption{Domain decomposition}
    \end{figure} 
\end{center}
 The idea is that the minimum defining $H$ will be led by $h(1-q,-p)$ for large positive velocities occurring in \textbf{1} and by $h(q,p)$ for large negative velocities occurring in \textbf{3}. As said in the introduction large velocities have to be considered compared to the distance boundary. Given the expression of $h$, it seems that the proper scaling to be considered are $p/q^{1/3}$ or $-p/(1-q)^{1/3}$. The domain \textbf{2} is therefore delimited by the curves of equations $p=-3q^{1/3}/t$ and $p=3(1-q)^{1/3}/t$. In \textbf{2}, we shall prove directly the lower bound $H$ where we shall use the fact that \textbf{2} is a bounded set in $(q,p)$.
 
 We first prove the lower-bound until some time $t_0>0$ and then generalize it in Proposition~\ref{lower-bound 2} for any time $t>0$. Therefore, we prove here the existence of a constant $c'_t>0$ depending on $t>0$ such that for all $t\in(0,t_0]$,
\begin{itemize}
    \item $\forall (q,p)\in$ \textbf{1}, $\mathbb{P}_{(q,p)}(\widehat{\tau}_\partial>t)\geq c'_th(1-q,-p)$.
    \item $\forall (q,p)\in$ \textbf{2}, $\mathbb{P}_{(q,p)}(\widehat{\tau}_\partial>t)\geq c'_t(h(1-q,-p)\wedge h(q,p))$.
    \item $\forall (q,p)\in$ \textbf{3}, $\mathbb{P}_{(q,p)}(\widehat{\tau}_\partial>t)\geq c'_th(q,p)$,
\end{itemize}
thus leading to the lower-bound proof. The generalization for any $t>0$ is done later in this section. Besides, since \textbf{1} and \textbf{3} are symmetric through the transformation $(q,p)\mapsto(1-q,-p)$ it is sufficient to prove the first two inequalities. 
\begin{proposition}[Lower-bound]\label{prop:lower bound}
There exists $t_0>0$ such that for all $t\in(0,t_0]$ there exists a constant $c'_t>0$ such that
$$\forall (q,p)\in D,\qquad \mathbb{P}_{(q,p)}(\widehat{\tau}_\partial>t)\geq c'_t\,H(q,p).$$ 
\end{proposition}

This proposition will later be extended for all time $t>0$ in~Proposition~\ref{lower-bound 2}. The proof relies on the two following lemmas. 
\begin{lemma}\label{lower bound tau_0 tau_1}
For any $t>0$ there exists $c'_{t}>0$ such that 
$$\forall q\in(0,1), \forall p\in[-3q^{1/3}/t,3(1-q)^{1/3}/t],\qquad\mathbb{P}_{(q,p)}(\widehat{\tau}_0>t)\geq c'_{t}h(q,p).$$ 
\end{lemma}  
\begin{lemma}\label{lemma prob cond}
There exists $t_0>0$ such that for all $t\in(0,t_0]$ there exists a constant $\alpha_{t}>0$ such that    
$$\forall q\in(0,1), \forall p\in[-3q^{1/3}/t,3(1-q)^{1/3}/t],\qquad\mathbb{P}_{(q,p)}(\widehat{\tau}_\partial>t)\geq\alpha_{t}(\mathbb{P}_{(q,p)}(\widehat{\tau}_0>t)\wedge\mathbb{P}_{(q,p)}(\widehat{\tau}_1>t)).$$ 
\end{lemma}

Assuming Lemmas~\ref{lower bound tau_0 tau_1} and~\ref{lemma prob cond} are satisfied, let us prove Proposition~\ref{prop:lower bound}.
\begin{proof}[Proof of Proposition~\ref{prop:lower bound}] Let $t_0>0$ be as defined in Lemma~\ref{lemma prob cond} and let $t\in(0,t_0)$. 
Let $q\in(0,1)$ and $p\in\mathbb{R}$ such that $p\in[-3q^{1/3}/t,3(1-q)^{1/3}/t]$. Notice that for such $(q,p)$, the couple $(1-q,-p)$ also satisfies the hypothesis of Lemma~\ref{lower bound tau_0 tau_1}. As a result, applying Lemma~\ref{lower bound tau_0 tau_1} along with Remark~\ref{invariant proba} one obtains that
$$\mathbb{P}_{(q,p)}(\widehat{\tau}_1>t)\geq c'_{t}h(1-q,-p).$$
Therefore, it follows from Lemma~\ref{lemma prob cond} that for all $q\in(0,1)$ and $p\in[-3q^{1/3}/t,3(1-q)^{1/3}/t]$,
\begin{equation}\label{eq:apply lemmata}
    \mathbb{P}_{(q,p)}(\widehat{\tau}_\partial>t)\geq\alpha_{t}c'_t(h(q,p)\wedge h(1-q,-p))=\alpha_{t}c'_tH(q,p).
\end{equation}
 
Thus, it remains to prove~\eqref{eq:apply lemmata} for $q\in(0,1)$ and $p/q^{1/3}<-3/t$ or $p/(1-q)^{1/3}>3/t$. We start with the case $p/q^{1/3}<-3/t$. 

Let $q\in(0,1)$ and $p/q^{1/3}<-3/t$. Let 
$$\widehat{\pi}_{t}:=\inf\{s>0:\widehat{p}_s/\widehat{q}_s^{1/3}\geq-3/t \}.$$ Necessarily, $\widehat{\tau}_0\wedge\widehat{\pi}_{t}\leq t/2$ almost-surely. Otherwise, integrating $\widehat{p}_s/\widehat{q}_s^{1/3}=(\mathrm{d}\widehat{q}_s/\mathrm{d}s)/\widehat{q}_s^{1/3}$ for $s\in[0,t/2]$, one obtains that $\widehat{q}_{t/2}\leq0$ contradicting that $\widehat{\tau}_0>t/2$. Besides, the continuity of the trajectories ensures that $\widehat{\tau}_1>\widehat{\pi}_t$ almost-surely. Therefore, applying the strong Markov property at the stopping time $\widehat{\pi}_t$,
\begin{align*}
    \mathbb{P}_{(q,p)}(\widehat{\tau}_\partial>t)&=\mathbb{E}_{(q,p)}\left[\mathbb{1}_{\widehat{\tau}_\partial>\widehat{\pi}_t}\mathbb{P}_{(\widehat{q}_{\widehat{\pi}_t},\widehat{p}_{\widehat{\pi}_t})}\left(\widehat{\tau}_\partial>t-s\right)\vert_{s=\widehat{\pi}_t}\right]\\ 
    &\geq\mathbb{E}_{(q,p)}\left[\mathbb{1}_{\widehat{\tau}_0>\widehat{\pi}_t}\mathbb{P}_{(\widehat{q}_{\widehat{\pi}_t},\widehat{p}_{\widehat{\pi}_t})}\left(\widehat{\tau}_\partial>t\right)\right].\numberthis\label{eq5}
\end{align*} 
On the event $\{\widehat{\tau}_0>\widehat{\pi}_t\}$, $\widehat{q}_{\widehat{\pi}_t}\in(0,1)$ and $\widehat{p}_{\widehat{\pi}_t}/\widehat{q}_{\widehat{\pi}_t}^{1/3}=-3/t$. Therefore, applying~\eqref{eq:apply lemmata} to $\mathbb{P}_{(\widehat{q}_{\widehat{\pi}_t},\widehat{p}_{\widehat{\pi}_t})}\left(\widehat{\tau}_\partial>t\right)$ in the expectation above, one has
\begin{equation}\label{utilisation lemma compact}
    \mathbb{P}_{(\widehat{q}_{\widehat{\pi}_t},\widehat{p}_{\widehat{\pi}_t})}\left(\widehat{\tau}_\partial>t\right)\geq\alpha_{t}c'_t(h(\widehat{q}_{\widehat{\pi}_t},\widehat{p}_{\widehat{\pi}_t})\wedge h(1-\widehat{q}_{\widehat{\pi}_t},-\widehat{p}_{\widehat{\pi}_t})). 
\end{equation}
By definition of $h$ and $\widehat{\pi}_t$,
\begin{equation}\label{eq12}
    h(\widehat{q}_{\widehat{\pi}_t},\widehat{p}_{\widehat{\pi}_t})=\widehat{q}_{\widehat{\pi}_t}^{1/6}g\left(\frac{\widehat{p}_{\widehat{\pi}_t}}{\widehat{q}_{\widehat{\pi}_t}^{1/3}}\right)=\widehat{q}_{\widehat{\pi}_t}^{1/6}g\left(-\frac{3}{t}\right).
\end{equation}
Besides, since $g$ is positive and satisfies the asymptotics in Remark~\ref{asymptotics g}, there exists $\beta>0$ such that for all $z\geq0$, $g(z)\geq\beta\sqrt{z}$. As a result,
\begin{align*}
    h(1-\widehat{q}_{\widehat{\pi}_t},-\widehat{p}_{\widehat{\pi}_t})&=(1-\widehat{q}_{\widehat{\pi}_t})^{1/6}g\left(-\frac{\widehat{p}_{\widehat{\pi}_t}}{(1-\widehat{q}_{\widehat{\pi}_t})^{1/3}}\right)\\
    &=(1-\widehat{q}_{\widehat{\pi}_t})^{1/6}g\left(\frac{3}{t}\left(\frac{\widehat{q}_{\widehat{\pi}_t}}{1-\widehat{q}_{\widehat{\pi}_t}}\right)^{1/3}\right)\\
    &\geq\beta\sqrt{\frac{3}{t}}\widehat{q}_{\widehat{\pi}_t}^{1/6}=\frac{\beta}{g(-3/t)}\sqrt{\frac{3}{t}}h(\widehat{q}_{\widehat{\pi}_t},\widehat{p}_{\widehat{\pi}_t}).\numberthis\label{eq13}
\end{align*}
Hence, reinjecting~\eqref{eq12} and~\eqref{eq13} into~\eqref{utilisation lemma compact} one has the existence of a constant $\gamma_t>0$ such that, on the event $\{\widehat{\tau}_0>\widehat{\pi}_t\}$,
$$\mathbb{P}_{(\widehat{q}_{\widehat{\pi}_t},\widehat{p}_{\widehat{\pi}_t})}\left(\tau_\partial>t\right)\geq\gamma_th(\widehat{q}_{\widehat{\pi}_t},\widehat{p}_{\widehat{\pi}_t}).$$
Therefore, reinjecting in~\eqref{eq5} it follows that
\begin{align*}
    \mathbb{P}_{(q,p)}(\widehat{\tau}_\partial>t)&\geq\gamma_t\mathbb{E}_{(q,p)}\left[\mathbb{1}_{\widehat{\tau}_0>\widehat{\pi}_t}h(\widehat{q}_{\widehat{\pi}_t},\widehat{p}_{\widehat{\pi}_t})\right]\\ 
    &=\gamma_t\mathbb{E}_{(q,p)}\left[h(\widehat{q}_{\widehat{\pi}_t\wedge\widehat{\tau}_0},\widehat{p}_{\widehat{\pi}_t\wedge\widehat{\tau}_0})\right]=\gamma_th(q,p) 
\end{align*} 
using Lemma~\ref{lemma martingale} and Doob's optional sampling theorem since $\widehat{\pi}_t\wedge\widehat{\tau}_0\leq t/2$ almost-surely and since $h$ vanishes continuously on the set $\{(0,p):p\leq0\}$, which follows from the asymptotics in Remark~\ref{asymptotics g}.

Now let us take $q\in(0,1)$ and $p/(1-q)^{1/3}>3/t$. In this case, since $(1-q,-p)$ satisfies the hypothesis of the previous case, we immediately have that $\mathbb{P}_{(1-q,-p)}(\widehat{\tau}_\partial>t)\geq\gamma_th(1-q,-p).$
Using Remark~\ref{invariant proba}, we deduce that for $q\in(0,1)$ and $p/(1-q)^{1/3}>3/t$
$$\mathbb{P}_{(q,p)}(\widehat{\tau}_\partial>t)\geq\gamma_th(1-q,-p).$$
Hence the proof for all $q\in(0,1)$ and $p\in\mathbb{R}$.
\end{proof}
\textcolor{black}{Before proving Lemmas~\ref{lower bound tau_0 tau_1} and~\ref{lemma prob cond}, let us state and prove the following lemma.}
\textcolor{black}{\begin{lemma}\label{lem:intermediary lemma}
Let $(q_n,p_n)_{n\geq0}$ be a sequence in $D$ such that
$$p_n\underset{n\rightarrow\infty}{\longrightarrow}0\,,\qquad\limsup_{n\rightarrow\infty}p_n/q_n^{1/3}=+\infty.$$
Then, there exists a constant $C>0$ such that for all $t>0$,
$$\liminf_{n\rightarrow\infty}\frac{\mathbb{P}_{(q_n,p_n)}(\widehat{\tau}_0>t)}{h(q_n,p_n)}\geq \frac{C}{t^{1/4}}.$$
\end{lemma}}

\begin{proof} Up to considering an appropriate subsequence we can assume that $p_n>0$ for all $n\geq0$ and $p_n/q_n^{1/3}\underset{n\rightarrow\infty}{\longrightarrow}+\infty$. In particular, it is easy to see that $q_n\rightarrow0$ when $n\rightarrow\infty$. 

\textbf{Step 1}: Let $\widehat{\pi}_0:=\inf\{s>0:\widehat{p}_s=0\}$. Let us start by showing that for any constant $\alpha_t>0$, there exists $\beta_t>0$ such that for $n\geq0$ large enough,
\begin{equation}\label{eq:asymptotics proba h(q_n,p_n)}
\mathbb{P}_{(q_n,p_n)}(\widehat{q}_{\widehat{\pi}_0}>\alpha_t)\leq\beta_t\,p_n.
\end{equation}
The explicit expression of the law of $\widehat{q}_{\widehat{\pi}_0}$ was given in Lachal's work in~\cite[Equation 6]{lachal1}. Therefore, if we take $N\geq1$ large enough such that for all $n\geq N$, $q_n\in(0,\alpha_t/2)$ then 
\begin{align*}
    \mathbb{P}_{(q_n,p_n)}(\widehat{q}_{\widehat{\pi}_0}>\alpha_t)&=\int_{\alpha_t}^\infty\frac{\Gamma(2/3)}{\pi 2^{2/3}3^{1/6}}\frac{p_n}{\vert z-q_n\vert^{4/3}}\mathrm{e}^{-2p_n^3/9\vert z-q_n\vert}\mathrm{d}z\nonumber\\
    &=\frac{\Gamma(2/3)}{\pi 2^{2/3}3^{1/6}}\int_0^{1/(\alpha_t-q_n)}p_n u^{-2/3}\mathrm{e}^{-2up_n^3/9}\mathrm{d}u\nonumber\\
    &\leq\frac{\Gamma(2/3)}{\pi 2^{2/3}3^{1/6}} p_n\int_0^{2/\alpha_t}  u^{-2/3}\mathrm{d}u,
\end{align*}
hence the proof.

\textbf{Step 2}: Let us conclude the proof using~\eqref{eq:asymptotics proba h(q_n,p_n)}. By the strong Markov property at $\widehat{\pi}_0$,
\begin{align}
\mathbb{P}_{(q_n,p_n)}(\widehat{\tau}_0>t)&\geq\mathbb{P}_{(q_n,p_n)}(\widehat{\tau}_0>t,\widehat{\pi}_0\leq t\land\widehat{\tau}_1)\nonumber\\ 
&=\mathbb{E}_{(q_n,p_n)}\left[\mathbb{1}_{\widehat{\pi}_0\leq t\land\widehat{\tau}_1}\mathbb{P}_{(\widehat{q}_{\widehat{\pi}_0},\widehat{p}_{\widehat{\pi}_0})}(\widehat{\tau}_0>t-s)\vert_{s=\widehat{\pi}_0}\right]\nonumber\\
&\geq\mathbb{E}_{(q_n,p_n)}\left[\mathbb{1}_{\widehat{\pi}_0\leq t\land\widehat{\tau}_1}\mathbb{P}_{(\widehat{q}_{\widehat{\pi}_0},\widehat{p}_{\widehat{\pi}_0})}(\widehat{\tau}_0>t)\right]\label{eq:ineq1 P tau_0)}
\end{align}
In addition, by Remark~\ref{rmk scale change},
\begin{equation}\label{eq:scaling proba q_pi_0}
\mathbb{P}_{(\widehat{q}_{\widehat{\pi}_0},\widehat{p}_{\widehat{\pi}_0})}(\widehat{\tau}_0>t)=\mathbb{P}_{(1,0)}(\widehat{\tau}_0>t/\widehat{q}_{\widehat{\pi}_0}^{2/3}).
\end{equation}
Given the limit~\eqref{long time asymptotics}, we deduce the existence of constants $c,M>0$ such that for all $t\geq M$,
$$\mathbb{P}_{(1,0)}(\widehat{\tau}_0>t)\geq\frac{c}{t^{1/4}}.$$
Consequently, by~\eqref{eq:ineq1 P tau_0)} and~\eqref{eq:scaling proba q_pi_0},
\begin{align*}
\mathbb{P}_{(q_n,p_n)}(\widehat{\tau}_0>t)&\geq\frac{c}{t^{1/4}}\mathbb{E}_{(q_n,p_n)}\left[\mathbb{1}_{\widehat{\pi}_0\leq t\land\widehat{\tau}_1}\mathbf{1}_{\widehat{q}_{\widehat{\pi}_0}\leq(t/M)^{3/2}}\widehat{q}_{\widehat{\pi}_0}^{1/6}\right]\\
&=\frac{c}{t^{1/4}g(0)}\mathbb{E}_{(q_n,p_n)}\left[\mathbb{1}_{\widehat{\pi}_0\leq t\land\widehat{\tau}_1}\mathbf{1}_{\widehat{q}_{\widehat{\pi}_0}\leq(t/M)^{3/2}}h(\widehat{q}_{\widehat{\pi}_0},\widehat{p}_{\widehat{\pi}_0})\right]\\
&\geq\frac{c}{t^{1/4}g(0)}\mathbb{E}_{(q_n,p_n)}\left[\mathbb{1}_{\widehat{\pi}_0\leq t\land\widehat{\tau}_1}h(\widehat{q}_{\widehat{\pi}_0},\widehat{p}_{\widehat{\pi}_0})\right]-\frac{c}{t^{1/4}}\mathbb{P}_{(q_n,p_n)}(\widehat{q}_{\widehat{\pi}_0}>(t/M)^{3/2}).
\end{align*}
Furthermore, by Lemma~\ref{lemma martingale},
\begin{equation}
\mathbb{E}_{(q_n,p_n)}\left[\mathbb{1}_{\widehat{\pi}_0\leq t\land\widehat{\tau}_1}h(\widehat{q}_{\widehat{\pi}_0},\widehat{p}_{\widehat{\pi}_0})\right]=h(q_n,p_n)-\mathbb{E}_{(q_n,p_n)}\left[\mathbb{1}_{\widehat{\pi}_0>t\land\widehat{\tau}_1}h(\widehat{q}_{t\land\widehat{\tau}_1},\widehat{p}_{t\land\widehat{\tau}_1})\right].\numberthis\label{eq:lb1}
\end{equation}
By the Hölder inequality with exponents $4/3$ and $4$ and considering the asymptotics of $g$ in Remark~\ref{asymptotics g} one has that 
\begin{align*} \mathbb{E}_{(q_n,p_n)}\left[\mathbb{1}_{\widehat{\pi}_0>t\land\widehat{\tau}_1}h(\widehat{q}_{t\land\widehat{\tau}_1},\widehat{p}_{t\land\widehat{\tau}_1})\right]&\leq\mathbb{P}_{(q_n,p_n)}(\widehat{\pi}_0>t\land\widehat{\tau}_1)^{3/4}\mathbb{E}_{(q_n,p_n)}\left[h(\widehat{q}_{t\land\widehat{\tau}_1},\widehat{p}_{t\land\widehat{\tau}_1})^4\right]^{1/4}\\
    &\leq\mathbb{P}_{(q_n,p_n)}(\widehat{\pi}_0>t\land\widehat{\tau}_1)^{3/4}\mathbb{E}_{(q_n,p_n)}\left[\sup_{s\in[0,t]}C(1+|\widehat{p}_s|^2)\right]^{1/4}\\
    &\leq C^{1/4}\mathbb{P}_{(q_n,p_n)}(\widehat{\pi}_0>t\land\widehat{\tau}_1)^{3/4}\mathbb{E}\left[1+p_n^2+2p_n\sup_{s\in[0,t]}|B_s|+\sup_{s\in[0,t]}|B_s|^2\right]^{1/4}\\
    &\leq C^{1/4}\mathbb{P}_{(q_n,p_n)}(\widehat{\pi}_0>t\land\widehat{\tau}_1)^{3/4}\left(1+p^2_n+\frac{2\sqrt{2t}}{\sqrt{\pi}}p_n+2t\frac{\sqrt{2}}{\sqrt{\pi}}\right)^{1/4}\numberthis\label{eq:lb2}.
\end{align*}
In addition
$$\mathbb{P}_{(q_n,p_n)}(\widehat{\pi}_0>t\land\widehat{\tau}_1)\leq\mathbb{P}_{(q_n,p_n)}(\widehat{\pi}_0>t)+\mathbb{P}_{(q_n,p_n)}(\widehat{\pi}_0>\widehat{\tau}_1).$$
Since $p_n>0$ for all $n\geq0$,
\begin{align*}
    \mathbb{P}_{(q_n,p_n)}(\widehat{\pi}_0>t)&=\mathbb{P}\big(\inf_{s\in[0,t]}B_s>-p_n\big)\\
    &=\mathbb{P}\big(\vert B_t\vert< p_n\big)\\
    &=\mathbb{P}\big(\vert B_1\vert< p_n/\sqrt{t}\big)\\
    &\leq \frac{\sqrt{2}}{\sqrt{\pi t}}p_n.
\end{align*}
Also,
$$\mathbb{P}_{(q_n,p_n)}(\widehat{\pi}_0>\widehat{\tau}_1)=\mathbb{P}_{(q_n,p_n)}(\widehat{q}_{\widehat{\pi}_0}>1).$$
Consequently, by~\eqref{eq:asymptotics proba h(q_n,p_n)}, there exists a constant $C'_t>0$ such that for all $n\geq N$,
\begin{equation}\label{eq:control of pi_0 probability}
    \mathbb{P}_{(q_n,p_n)}(\widehat{\pi}_0>t\land\widehat{\tau}_1)\leq C'_tp_n.
\end{equation}
It follows then from~\eqref{eq:lb1},~\eqref{eq:lb2},~\eqref{eq:control of pi_0 probability} and the asymptotics of $g$ that
$$\liminf_{n\rightarrow\infty}\frac{\mathbb{P}_{(q_n,p_n)}(\widehat{\tau}_0>t)}{h(q_n,p_n)}\geq \frac{c}{t^{1/4}g(0)}$$
which concludes the proof.
\end{proof}

Let us now prove Lemma~\ref{lower bound tau_0 tau_1}.
\begin{proof}[Proof of Lemma~\ref{lower bound tau_0 tau_1}]
Let $t>0$. Let us take sequences $(q_n)_{n\geq0}$, $(p_n)_{n\geq0}$ such that for all $n\geq0$, $q_n\in(0,1)$ and $p_n\in[-3q_n^{1/3}/t,3(1-q_n)^{1/3}/t]$ and prove that
$$\liminf_{n\rightarrow\infty}\mathbb{P}_{(q_n,p_n)}(\widehat{\tau}_0>t)/h(q_n,p_n)>0.$$ 

Both sequences $(q_n)_{n\geq0}$ and $(p_n)_{n\geq0}$ are bounded. Therefore, up to extracting appropriate subsequences, we can assume that they both converge. For $n\geq0$, using Remark~\ref{rmk scale change},
\begin{equation}\label{scaling eq1}
    \mathbb{P}_{(q_n,p_n)}(\widehat{\tau}_0>t)=\mathbb{P}_{(1,p_n/q_n^{1/3})}(\widehat{\tau}_0>t/q_n^{2/3}).
\end{equation}
First assume that $\limsup_{n\rightarrow\infty}p_n/q_n^{1/3}<\infty$, then the sequence $(p_n/q_n^{1/3})_{n\geq0}$ is bounded from above by a constant $M>0$. Therefore, since for all $n\geq0$,  $-3/t\leq p_n/q_n^{1/3}\leq M$,
$$\frac{\mathbb{P}_{(q_n,p_n)}(\widehat{\tau}_0>t)}{h(q_n,p_n)}\geq \frac{\mathbb{P}_{(1,-3/t)}(\widehat{\tau}_0>t/q_n^{2/3})}{q_n^{1/6}g(M)}.$$ 
If $q_n\underset{n\rightarrow\infty}{\longrightarrow}q_\infty>0$, then the term in the right-hand side of the inequality above is clearly bounded from below by a positive time-dependent constant. If $q_n\underset{n\rightarrow\infty}{\longrightarrow}0$, using the long-time asymptotics in~\eqref{long time asymptotics} from~\cite{isozaki1994,GP} one has the existence of a constant $\alpha_t>0$ such that
$$\frac{\mathbb{P}_{(1,-3/t)}(\widehat{\tau}_0>t/q_n^{2/3})}{q_n^{1/6}}\underset{n\rightarrow\infty}{\sim}\frac{\alpha_tq_n^{1/6}}{q_n^{1/6}}=\alpha_t>0.$$ 
Assume now that 
$\limsup_{n\rightarrow\infty}p_n/q_n^{1/3}=\infty$, up to taking a subsequence we can assume that for all $n\geq0$, $p_n>0$ and that $p_n/q_n^{1/3}\underset{n\rightarrow\infty}{\longrightarrow}\infty.$

First, consider the case $p_n\underset{n\rightarrow\infty}{\longrightarrow}p_\infty>0$. Then up to taking a subsequence we can assume that $p_n\geq p_\infty/2$ for all $n\geq0$. Besides, $\mathbb{P}_{(q,p)}$ almost-surely for all $s\geq0$,
$$\widehat{q}_s=q+\int_0^s(B_r+p)\mathrm{d}r,$$
one has that 
$$\mathbb{P}_{(q_n,p_n)}(\widehat{\tau}_0>t)\geq\mathbb{P}\bigg(\inf_{s\in[0,t]}B_s\geq-p_n\bigg)\geq\mathbb{P}\bigg(\inf_{s\in[0,t]}B_s\geq-p_\infty/2\bigg).$$
Therefore, by the asymptotics of $g$ in Remark~\ref{asymptotics g},
 \begin{align*}
     \frac{\mathbb{P}_{(q_n,p_n)}(\widehat{\tau}_0>t)}{h(q_n,p_n)}&\geq\frac{\mathbb{P}\big(\inf_{s\in[0,t]}B_s\geq-p_\infty/2\big)}{q_n^{1/6}g(p_n/q_n^{1/3})}\\
     &\underset{n\rightarrow\infty}{\sim}\frac{\mathbb{P}\big(\inf_{s\in[0,t]}B_s\geq-p_\infty/2\big)}{\sqrt{p_\infty}}>0.
 \end{align*} 
\textcolor{black}{The remaining case $\limsup_{n\rightarrow\infty}p_n/q_n^{1/3}=\infty$ and $p_n\underset{n\rightarrow\infty}{\longrightarrow}0$ was already answered in Lemma~\ref{lem:intermediary lemma}, hence the proof.}
\end{proof}

Let us now prove Lemma~\ref{lemma prob cond}.
\begin{proof}[Proof of Lemma~\ref{lemma prob cond}]
Let $t>0$. Since $\mathbb{P}_{(q,p)}(\widehat{\tau}_0>t)$ continuously vanishes on the set $\{(0,p):p\leq0\}$ and $\mathbb{P}_{(q,p)}(\widehat{\tau}_1>t)$ continuously vanishes on the set $\{(1,p):p\geq0\}$, it is enough to prove that for any sequences $(q_n)_{n\geq0}$, $(p_n)_{n\geq0}$ such that
$$q_n\underset{n\rightarrow\infty}{\longrightarrow}0,\qquad p_n\underset{n\rightarrow\infty}{\longrightarrow}p_\infty\leq0$$
and
$$\forall n\geq0,\qquad q_n\in(0,1),\qquad p_n\in[-3q_n^{1/3}/t,3(1-q_n)^{1/3}/t],$$
we have the existence of $t_0>0$ such that for all $t\in(0,t_0]$,
$$\liminf_{n\rightarrow\infty}\frac{\mathbb{P}_{(q_n,p_n)}(\widehat{\tau}_\partial>t)}{\mathbb{P}_{(q_n,p_n)}(\widehat{\tau}_0>t)}>0\quad\text{and}\quad\liminf_{n\rightarrow\infty}\frac{\mathbb{P}_{(1-q_n,-p_n)}(\widehat{\tau}_\partial>t)}{\mathbb{P}_{(1-q_n,-p_n)}(\widehat{\tau}_1>t)}>0.$$
Using the invariance in Remark~\ref{invariant proba} this is equivalent to the proof of the first liminf. Moreover, for $n\geq0$,
\begin{align*}
    \mathbb{P}_{(q_n,p_n)}(\widehat{\tau}_\partial>t)&=\frac{\mathbb{P}_{(q_n,p_n)}(\widehat{\tau}_\partial>t)}{\mathbb{P}_{(q_n,p_n)}(\widehat{\tau}_0>t)}\mathbb{P}_{(q_n,p_n)}(\widehat{\tau}_0>t)\\
&=\left(1-\frac{\mathbb{P}_{(q_n,p_n)}(\widehat{\tau}_0>t,\widehat{\tau}_1\leq t)}{\mathbb{P}_{(q_n,p_n)}(\widehat{\tau}_0>t)}\right)\mathbb{P}_{(q_n,p_n)}(\widehat{\tau}_0>t).
\end{align*}  
Our aim here is to obtain appropriate upper-bounds of $\frac{\mathbb{P}_{(q_n,p_n)}(\widehat{\tau}_0>t,\widehat{\tau}_1\leq t)}{\mathbb{P}_{(q_n,p_n)}(\widehat{\tau}_0>t)}$ in order to obtain a  positive lower-bound for $\mathbb{P}_{(q_n,p_n)}(\widehat{\tau}_\partial>t)/\mathbb{P}_{(q_n,p_n)}(\widehat{\tau}_0>t)$.

Let $\widehat{\pi}_0:=\inf\{s\geq0:\widehat{p}_s=0\}$. If $q\in(0,1)$ and $p\leq0$, the continuity of the velocity ensures that $\widehat{\pi}_0<\widehat{\tau}_1$ $\mathbb{P}_{(q,p)}$ almost-surely. Therefore, if $p_n\leq0$, applying the strong Markov property at $\widehat{\pi}_0$ one obtains that
\begin{align*}
    \mathbb{P}_{(q_n,p_n)}(\widehat{\tau}_0>t,\widehat{\tau}_1\leq t)&\leq\mathbb{E}_{(q_n,p_n)}\left[\mathbb{1}_{\widehat{\tau}_0>\widehat{\pi}_0}\mathbb{P}_{(\widehat{q}_{\widehat{\pi}_0},\widehat{p}_{\widehat{\pi}_0})}(\widehat{\tau}_0>t-s,\widehat{\tau}_1\leq t-s)\vert_{s=\widehat{\pi}_0}\right]\\
    &\leq\mathbb{E}_{(q_n,p_n)}\left[\mathbb{1}_{\widehat{\tau}_0>\widehat{\pi}_0}\mathbb{P}_{(\widehat{q}_{\widehat{\pi}_0},0)}(\widehat{\tau}_0>\widehat{\tau}_1)\right].\numberthis\label{Markov sigma_0} 
\end{align*}  
Besides, thanks to the work of Lachal in~\cite[Equation 4]{lachal2} one has an exact expression of $\mathbb{P}_{(q,0)}(\widehat{\tau}_0>\widehat{\tau}_1)$ for $q\in(0,1)$ where $F,\beta$ are respectively the hypergeometric and the beta function,
\begin{align*}
\mathbb{P}_{(q,0)}(\widehat{\tau}_0>\widehat{\tau}_1)&=6\frac{\Gamma(1/3)}{\Gamma(1/6)^2}q^{1/6}F(1/6,5/6,7/6,q)\\    
&=6\frac{\Gamma(1/3)}{\Gamma(1/6)^2}q^{1/6}\frac{\int_0^1x^{-1/6}(1-x)^{-2/3}(1-qx)^{-1/6}\mathrm{d}x}{\beta(5/6,1/3)}\\
&\leq 6\frac{\Gamma(1/3)}{\Gamma(1/6)^2}q^{1/6}\frac{\beta(5/6,1/6)}{\beta(5/6,1/3)}\\
&=q^{1/6},
\end{align*}
using the fact that $\beta(a,b)=\Gamma(a)\Gamma(b)/\Gamma(a+b)$ along with $\Gamma(a+1)=a\Gamma(a)$. As a result, reinjecting in~\eqref{Markov sigma_0}, one has 
\begin{align*}
    \mathbb{P}_{(q_n,p_n)}(\widehat{\tau}_0>t,\widehat{\tau}_1\leq t)&\leq\mathbb{E}_{(q_n,p_n)}\left[\mathbb{1}_{\widehat{\tau}_0>\widehat{\pi}_0}q_{\widehat{\pi}_0}^{1/6}\right]\\
    &=\frac{1}{g(0)}\mathbb{E}_{(q_n,p_n)}\bigg[\mathbb{1}_{\widehat{\tau}_0>\widehat{\pi}_0}\underbrace{\widehat{q}_{\widehat{\pi}_0}^{1/6}g(0)}_{=h(\widehat{q}_{\widehat{\pi}_0},\widehat{p}_{\widehat{\pi}_0})}\bigg]\\
    &=\frac{1}{g(0)}\mathbb{E}_{(q_n,p_n)}\bigg[h(\widehat{q}_{\widehat{\pi}_0\wedge\widehat{\tau}_0},\widehat{p}_{\widehat{\pi}_0\wedge\widehat{\tau}_0})\bigg] 
\end{align*}
since $h$ vanishes on the set $\{(0,p):p\leq0\}$. Besides, since $p_n\leq0$, $\mathbb{P}_{(q_n,p_n)}$ almost-surely,
$$h(\widehat{q}_{\widehat{\pi}_0\wedge\widehat{\tau}_0},\widehat{p}_{\widehat{\pi}_0\wedge\widehat{\tau}_0})\leq q_{\widehat{\pi}_0\wedge\widehat{\tau}_0}^{1/6}g(0)\leq g(0).$$
Therefore, using Lemma~\ref{lemma martingale} it follows from Doob's optional sampling theorem that $\mathbb{E}_{(q_n,p_n)}\bigg[h(\widehat{q}_{\widehat{\pi}_0\wedge\widehat{\tau}_0},\widehat{p}_{\widehat{\pi}_0\wedge\widehat{\tau}_0})\bigg]=h(q_n,p_n)$. As a result,
\begin{equation}\label{eq:majo1}
\mathbb{P}_{(q_n,p_n)}(\widehat{\tau}_0>t,\widehat{\tau}_1\leq t)\leq\frac{h(q_n,p_n)}{g(0)}.
\end{equation}  
However, in the case $p_n>0$, one cannot ensure that $\widehat{\pi}_0<\widehat{\tau}_1$ almost-surely. Therefore,
\begin{align*}
    &\mathbb{P}_{(q_n,p_n)}(\widehat{\tau}_0>t,\widehat{\tau}_1\leq t)\\
    &\leq\mathbb{E}_{(q_n,p_n)}\left[\mathbb{1}_{\widehat{\tau}_0\wedge\widehat{\tau}_1>\widehat{\pi}_0}\mathbb{1}_{\widehat{\pi}_0\leq t}\mathbb{P}_{(\widehat{q}_{\widehat{\pi}_0},0)}(\widehat{\tau}_0>t-s,\widehat{\tau}_1\leq t-s)\vert_{s=\widehat{\pi}_0}\right]+\mathbb{P}_{(q_n,p_n)}(\widehat{\tau}_0>t,\widehat{\tau}_1\leq t,\widehat{\pi}_0\geq\widehat{\tau}_1)\\
    &\leq \frac{1}{g(0)}\mathbb{E}_{(q_n,p_n)}\bigg[h(\widehat{q}_{\widehat{\pi}_0\wedge\widehat{\tau}_0\wedge t},\widehat{p}_{\widehat{\pi}_0\wedge\widehat{\tau}_0\wedge t})\bigg]+\mathbb{P}_{(q_n,p_n)}(\widehat{\pi}_0\geq\widehat{\tau}_1)\\
    &\leq\frac{h(q_n,p_n)}{g(0)}+\mathbb{P}_{(q_n,p_n)}(\widehat{\pi}_0\geq\widehat{\tau}_1).\numberthis\label{eq:majo2}
\end{align*}  
 Furthermore, one has that
\begin{equation}\label{eq:majo5}
    \mathbb{P}_{(q_n,p_n)}(\widehat{\pi}_0\geq\widehat{\tau}_1)\leq\mathbb{P}_{(q_n,p_n)}(\widehat{q}_{\widehat{\pi}_0}>1).
\end{equation}
 Since $q_n\underset{n\rightarrow\infty}{\longrightarrow}0$, up to taking a subsequence, we can assume that $q_n\in(0,1/2)$ for $n\geq0$. In addition, reinjecting the explicit expression of the law of $\widehat{q}_{\widehat{\pi}_0}$, which follows from Lachal's work in~\cite[Equation 6]{lachal1}, one has, as in the proof of Lemma~\ref{lem:intermediary lemma}, that
 \begin{equation}\label{eq:majo3}
\mathbb{P}_{(q_n,p_n)}(\widehat{q}_{\widehat{\pi}_0}>1)\leq\frac{\Gamma(2/3)}{\pi 2^{2/3}3^{1/6}} p_n\int_0^{2}  u^{-2/3}\mathrm{d}u.
 \end{equation}  
Let us now compute an upper-bound for $\limsup_{n\rightarrow\infty}\frac{\mathbb{P}_{(q_n,p_n)}(\widehat{\tau}_0>t,\widehat{\tau}_1\leq t)}{\mathbb{P}_{(q_n,p_n)}(\widehat{\tau}_0>t)}$. 

Assuming there is some $N\geq0$ such that for all $n\geq N$, $p_n\leq0$. Then using~\eqref{eq:majo1} it follows that for $n\geq N$, 
\begin{align*}
    \frac{\mathbb{P}_{(q_n,p_n)}(\widehat{\tau}_0>t,\widehat{\tau}_1\leq t)}{\mathbb{P}_{(q_n,p_n)}(\widehat{\tau}_0>t)}&\leq\frac{1}{g(0)}\frac{h(q_n,p_n)}{\mathbb{P}_{(q_n,p_n)}(\widehat{\tau}_0>t)}\\
    &=\frac{1}{g(0)}\frac{q_n^{1/6}g(p_n/q_n^{1/3})}{\mathbb{P}_{(1,p_n/q_n^{1/3})}(\widehat{\tau}_0>t/q_n^{2/3})}.
\end{align*}
Since $p_n/q_n^{1/3}\in[-3/t,0]$ for all $n\geq N$, up to taking an appropriate subsequence, one can assume that the sequence $(p_n/q_n^{1/3})_{n\geq N}$ converges to a limit $l\leq0$. 

Let $\epsilon>0$. Let $N'\geq1$ such that for all $n\geq N'$, $p_n/q_n^{1/3}\geq l-\epsilon$. Then for all $n\geq N'\lor N$, 
\begin{equation}\label{eq:computation p_n neg}
\frac{\mathbb{P}_{(q_n,p_n)}(\widehat{\tau}_0>t,\widehat{\tau}_1\leq t)}{\mathbb{P}_{(q_n,p_n)}(\widehat{\tau}_0>t)}\leq\frac{1}{g(0)}\frac{q_n^{1/6}g(p_n/q_n^{1/3})}{\mathbb{P}_{(1,l-\epsilon)}(\widehat{\tau}_0>t/q_n^{2/3})}. 
\end{equation}
It follows from~\eqref{long time asymptotics} that  $\mathbb{P}_{(1,l-\epsilon)}(\widehat{\tau}_0>t)\underset{t\rightarrow\infty}{\sim}Cg(l-\epsilon) t^{-1/4}$ where $C$ is a universal constant. Therefore,
$$\limsup_{n\rightarrow\infty}\frac{\mathbb{P}_{(q_n,p_n)}(\widehat{\tau}_0>t,\widehat{\tau}_1\leq t)}{\mathbb{P}_{(q_n,p_n)}(\widehat{\tau}_0>t)}\leq \frac{t^{1/4}}{Cg(0)}\limsup_{n\rightarrow\infty}\frac{g(p_n/q_n^{1/3})}{g(l-\epsilon)}=\frac{t^{1/4}}{Cg(0)}\frac{g(l)}{g(l-\epsilon)}.$$
Since the term in the left-hand side of the inequality above does not depend on $\epsilon$, one can take the limit $\epsilon\rightarrow0$ thus obtaining
\begin{equation}\label{eq:majo4}
    \limsup_{n\rightarrow\infty}\frac{\mathbb{P}_{(q_n,p_n)}(\widehat{\tau}_0>t,\widehat{\tau}_1\leq t)}{\mathbb{P}_{(q_n,p_n)}(\widehat{\tau}_0>t)}\leq \frac{t^{1/4}}{Cg(0)},
\end{equation}
which can be made as small as possible for $t\in(0,t_0)$ with $t_0>0$ small enough since $Cg(0)$ is a universal constant. 

Assume now that for all $n\geq0$, there exists $n'\geq n$ such that $p_{n'}>0$ then necessarily $p_n\underset{n\rightarrow\infty}{\longrightarrow}0$. Besides, up to taking an appropriate subsequence, one can assume that for all $n\geq0$, $p_n>0$. It follows from~\eqref{eq:majo2},~\eqref{eq:majo5} and~\eqref{eq:majo3} that there exists a universal constant $C'>0$ such that
$$\frac{\mathbb{P}_{(q_n,p_n)}(\widehat{\tau}_0>t,\widehat{\tau}_1\leq t)}{\mathbb{P}_{(q_n,p_n)}(\widehat{\tau}_0>t)}\leq\frac{1}{g(0)}\frac{h(q_n,p_n)}{\mathbb{P}_{(q_n,p_n)}(\widehat{\tau}_0>t)}+C'\frac{p_n}{\mathbb{P}_{(q_n,p_n)}(\widehat{\tau}_0>t)}.$$
First, consider the case $\limsup_{n\rightarrow\infty}p_n/q_n^{1/3}<\infty$. \textcolor{black}{Up to taking an appropriate subsequence, one can assume that the sequence $(p_n/q_n^{1/3})_{n\geq0}$ converges to a limit $l\geq0$. Similarly to the computation in~\eqref{eq:computation p_n neg}, we can show that for any $\epsilon>0$, for $n\geq0$ large enough,
\begin{align*}
\frac{\mathbb{P}_{(q_n,p_n)}(\widehat{\tau}_0>t,\widehat{\tau}_1\leq t)}{\mathbb{P}_{(q_n,p_n)}(\widehat{\tau}_0>t)}&\leq\frac{1}{g(0)}\frac{q_n^{1/6}g(p_n/q_n^{1/3})}{\mathbb{P}_{(1,l-\epsilon)}(\widehat{\tau}_0>t/q_n^{2/3})}+C'(l+\epsilon)\frac{q_n^{1/3}}{\mathbb{P}_{(1,l-\epsilon)}(\widehat{\tau}_0>t/q_n^{2/3})}\\
&\underset{n\rightarrow\infty}{\longrightarrow}\frac{t^{1/4}}{Cg(0)}\frac{g(l)}{g(l-\epsilon)}+\frac{C'(l+\epsilon)t^{1/4}}{Cg(l-\epsilon)}\lim_{n\rightarrow\infty}q_n^{1/6}\\
&\underset{\epsilon\rightarrow0}{\longrightarrow}\frac{t^{1/4}}{Cg(0)}.
\end{align*}
}
\textcolor{black}{As a result, the inequality~\eqref{eq:majo4} is still satisfied. In the case $\limsup_{n\rightarrow\infty}p_n/q_n^{1/3}=\infty$, applying Lemma~\ref{lem:intermediary lemma}, we obtain that the inequality~\eqref{eq:majo4} still holds, which concludes the proof of Lemma~\ref{lemma prob cond}.}
\end{proof} 
 
In order to extend the lower-bound obtained in Proposition~\ref{prop:lower bound} for any time $t>0$, let us first prove the following lemma which is stated for a multidimensional underdamped Langevin process.

\begin{lemma}[Control in a compact set]\label{lem:control in a compact}
Let $d\geq1$. Let $\mathcal{O}$ be a $\mathcal{C}^2$ bounded connected open set of $\mathbb{R}^d$. Let $D=\mathcal{O}\times\mathbb{R}^d$ and let $(q_t,p_t)_{t\geq0}$ be the process in $\mathbb{R}^d\times\mathbb{R}^d$ solution to
\begin{equation*}
  \left\{
    \begin{aligned}
        &\mathrm{d}q_t=p_t \mathrm{d}t , \\ &\mathrm{d}p_t=F(q_t)\mathrm{d}t-\gamma p_t\mathrm{d}t+\sigma\mathrm{d}B_t,
    \end{aligned}
\right.  
\end{equation*}
where $F\in\mathcal{C}^\infty(\mathbb{R}^d)$, $\gamma\in\mathbb{R}$ and $\sigma>0$. Let $\tau_\partial=\inf\{t>0:q_t\notin\mathcal{O}\}.$ Assume that there exist $t_0>0$, a function $H_{F,\gamma,\sigma}$ in $D$ such that for all $t\in(0,t_0]$, 
$$\mathbb{P}_{(q,p)}(\tau_\partial>t)\asymp H_{F,\gamma,\sigma}(q,p).$$
Then, there exists a compact set $K_0\subset D$ such that
$$\inf_{(q,p)\in D}\frac{\mathbb{P}_{(q,p)}((q_{t_0},p_{t_0})\in K_0,\tau_\partial>t_0)}{\mathbb{P}_{(q,p)}(\tau_\partial>t_0)}>0.$$ 
\end{lemma}

\begin{proof}
Let $t_0>0$ be defined as in the assumption. Let $K_{0}\subset D$ be any given compact set then for any $(q,p)\in D$,
\begin{equation}\label{decomp prob cond compact}
    \frac{\mathbb{P}_{(q,p)}((q_{t_0},p_{t_0})\in K_{0},\tau_\partial>t_0)}{\mathbb{P}_{(q,p)}(\tau_\partial>t_0)}=1-\frac{\mathbb{P}_{(q,p)}((q_{t_0},p_{t_0})\notin K_{0},\tau_\partial>t_0)}{\mathbb{P}_{(q,p)}(\tau_\partial>t_0)}.
\end{equation}
Besides, using the existence of a transition density $\mathrm{p}^D$ for the absorbed kernel, see~\cite[Theorem 2.20]{kFP}, along with the Chapman-Kolmogorov relation, one has that
$$\mathbb{P}_{(q,p)}((q_{t_0},p_{t_0})\notin K_{0},\tau_\partial>t_0)=\iiint_{D\times D\times K_{0}^c}\mathrm{p}^D_{t_0/3}((q,p),y)\mathrm{p}^D_{t_0/3}(y,z)\mathrm{p}^D_{t_0/3}(z,w)\mathrm{d}y\mathrm{d}z\mathrm{d}w.$$
Furthermore, $\mathrm{p}^D$ satisfies a Gaussian upper-bound, see~\cite[Theorem 2.19]{kFP} which ensures in particular that $\mathrm{p}^D_{t_0/3}$ is uniformly bounded on $D\times D$ by a constant $\alpha_{0}>0$. As a result,
\begin{align*}
    \mathbb{P}_{(q,p)}((q_{t_0},p_{t_0})\notin K_{0},\tau_\partial>t_0)&\leq\alpha_{0}\int_D\mathrm{p}^D_{t_0/3}((q,p),y)\mathrm{d}y\iint_{D\times K_0^c}\mathrm{p}^D_{t_0/3}(z,w)\mathrm{d}z\mathrm{d}w\\
    &=\alpha_{0}\mathbb{P}_{(q,p)}(\tau_\partial>t_0/3)\iint_{D\times K_0^c}\mathrm{p}^D_{t_0/3}(z,w)\mathrm{d}z\mathrm{d}w.
\end{align*}
Consequently,
\begin{equation}\label{proba conditionnelle double integ}
    \frac{\mathbb{P}_{(q,p)}((q_{t_0},p_{t_0})\notin K_{0},\tau_\partial>t_0)}{\mathbb{P}_{(q,p)}(\tau_\partial>t_0)}\leq\alpha_{0}\frac{\mathbb{P}_{(q,p)}(\tau_\partial>t_0/3)}{\mathbb{P}_{(q,p)}(\tau_\partial>t_0)}\iint_{D\times K_{0}^c}\mathrm{p}^D_{t_0/3}(z,w)\mathrm{d}z\mathrm{d}w.
\end{equation}
By assumption there exists a constant $c_{0}>0$ independent of $(q,p)\in D$ such that
$$\frac{\mathbb{P}_{(q,p)}(\tau_\partial>t_0/3)}{\mathbb{P}_{(q,p)}(\tau_\partial>t_0)}\leq c_{0}.$$
Furthermore, it was shown in~\cite[Lemma 3.1]{QSD1} that for any $t>0$, $\mathrm{p}^D_t\in\mathrm{L}^1(D\times D)$. Consequently, there exists a compact set $K_{0}\subset D$ large enough such that
$$\iint_{D\times K_{0}^c}\mathrm{p}^D_{t_0/3}(z,w)\mathrm{d}z\mathrm{d}w\leq\frac{1}{2\alpha_{0}c_{0}}.$$
The two inequalities above ensure that for $K_{0}$ defined as such, for all $(q,p)\in D$, 
$$\frac{\mathbb{P}_{(q,p)}((q_{t_0},p_{t_0})\in K_{0},\tau_\partial>t_0)}{\mathbb{P}_{(q,p)}(\tau_\partial>t_0)}\geq\frac{1}{2}.$$ 
\end{proof}
We are now able to prove Proposition~\ref{lower-bound 2}.
\begin{proposition}[Lower-bound]\label{lower-bound 2}
For all $t>0$, there exists a constant $c'_t>0$ such that
\begin{equation}\label{eq:lower bound all time}
    \forall (q,p)\in D,\qquad \mathbb{P}_{(q,p)}(\widehat{\tau}_\partial>t)\geq c'_t\,H(q,p).
\end{equation}
\end{proposition}
 
\begin{proof}[Proof of Proposition~\ref{lower-bound 2}]
Let $t_0>0$ be as defined in Proposition~\ref{prop:lower bound}. Using Propositions~\ref{prop:upper-bound} and~\ref{prop:lower bound} we can apply Lemma~\ref{lem:control in a compact} to obtain   
\begin{equation}\label{eq:inf proba conditionnelle}
c_0:=\inf_{(q,p)\in D}\frac{\mathbb{P}_{(q,p)}((\widehat{q}_{t_0},\widehat{p}_{t_0})\in K_{0},\widehat{\tau}_\partial>t_0)}{\mathbb{P}_{(q,p)}(\widehat{\tau}_\partial>t_0)}>0.   
\end{equation}
Let us now prove~\eqref{eq:lower bound all time}. In order to do that, we start by considering, for any integer $k\geq1$,  the probability \mbox{$\mathbb{P}_{(q,p)}(\widehat{\tau}_\partial>(k+1)t_0)$}. Using the Markov property one has
\begin{align*}
    \mathbb{P}_{(q,p)}(\widehat{\tau}_\partial>(k+1)t_0)&=\mathbb{E}_{(q,p)}\left[\mathbb{1}_{\widehat{\tau}_\partial>kt_0}\mathbb{P}_{(\widehat{q}_{kt_0},\widehat{p}_{kt_0})}(\widehat{\tau}_\partial>t_0)\vert\right]\\
    &\geq\mathbb{E}_{(q,p)}\left[\mathbb{1}_{\widehat{\tau}_\partial>kt_0,(\widehat{q}_{kt_0},\widehat{p}_{kt_0})\in K_0}\mathbb{P}_{(\widehat{q}_{kt_0},\widehat{p}_{kt_0})}(\widehat{\tau}_\partial>t_0)\right].
\end{align*}
Let $\delta_0:=\inf_{z\in K_0}\mathbb{P}_z(\widehat{\tau}_\partial>t_0)$ then $\delta_0>0$ since $K_0$ is a compact subset of $D$ and such probability is continuous and positive on $D$ by~\cite[Theorem 2.20]{kFP}. Therefore,
\begin{align*}
    \mathbb{P}_{(q,p)}(\widehat{\tau}_\partial>(k+1)t_0)&\geq\delta_0\mathbb{P}_{(q,p)}((\widehat{q}_{kt_0},\widehat{p}_{kt_0})\in K_0,\widehat{\tau}_\partial>kt_0)\\
    &=\delta_0\mathbb{E}_{(q,p)}\left[\mathbb{1}_{\widehat{\tau}_\partial>(k-1)t_0}\mathbb{P}_{(\widehat{q}_{(k-1)t_0},\widehat{p}_{(k-1)t_0})}((\widehat{q}_{t_0},\widehat{p}_{t_0})\in K_0,\widehat{\tau}_\partial>t_0)\right].
\end{align*}
As a result, by~\eqref{eq:inf proba conditionnelle},
\begin{align*}
    \mathbb{P}_{(q,p)}(\widehat{\tau}_\partial>(k+1)t_0)&\geq\delta_0c_0\mathbb{E}_{(q,p)}\left[\mathbb{1}_{\widehat{\tau}_\partial>(k-1)t_0}\mathbb{P}_{(\widehat{q}_{(k-1)t_0},\widehat{p}_{(k-1)t_0})}(\widehat{\tau}_\partial>t_0)\right]\\
    &=\delta_0c_0\mathbb{P}_{(q,p)}(\widehat{\tau}_\partial>kt_0).
\end{align*}
As a result, for all $k\geq1$,
$$\mathbb{P}_{(q,p)}(\widehat{\tau}_\partial>(k+1)t_0)\geq(\delta_0c_0)^k\mathbb{P}_{(q,p)}(\widehat{\tau}_\partial>t_0)\geq(\delta_0c_0)^kc'_{t_0}H(q,p)$$
by Proposition~\ref{prop:lower bound}.

Now let us take any $t>t_0$. Then there exists an integer $k\geq1$ and $s\in[0,t_0)$ such that $t=kt_0+s$. Since $\mathbb{P}_{(q,p)}(\widehat{\tau}_\partial>t)\geq\mathbb{P}_{(q,p)}(\widehat{\tau}_\partial>(k+1)t_0)$ and $k=\lfloor t/t_0\rfloor$, one has from the inequality above that
$$\mathbb{P}_{(q,p)}(\widehat{\tau}_\partial>t)\geq(\delta_0c_0)^{\lfloor t/t_0\rfloor}c'_{t_0}H(q,p)$$
which concludes the proof of Proposition~\ref{lower-bound 2}. 
\end{proof}   

Finally, Propositions~\ref{prop:upper-bound} and~\ref{lower-bound 2} conclude the proof of Proposition~\ref{prop:two-sided proba} which yields Proposition~\ref{prop:two-sided proba final}. Let us mention that applying Proposition~\ref{prop:first exit estimates implies two sided} to the process~\eqref{eq:int brownian motion with sigma} now ensures two-sided estimates for the transition density of the absorbed process~\eqref{eq:int brownian motion with sigma} which we shall use in the next section. Let us also recall that the proof of Proposition~\ref{prop:first exit estimates implies two sided} is completed in Section~\ref{sec:two-sided density}.

\subsection{\textcolor{black}{Underdamped Langevin process: proof of Proposition~\ref{prop:estimates first exit time Langevin}}}\label{sec:first exit time proba langevin}
 
\textcolor{black}{The objective of this section is to extend the estimates obtained in Proposition~\ref{prop:two-sided proba final} to the underdamped Langevin process~\eqref{eq:Langevin}. The strategy of the proof consists in proving the two-sided estimates for an auxiliary process detailed below. Using a Girsanov-type argument we are then able to deduce the two-sided estimates satisfied by~\eqref{eq:Langevin}.}

\textcolor{black}{Let us first define the function
$$\varphi:q\geq0\mapsto
\begin{cases}
-\frac{\sigma^2}{3q^{5/3}}\int_0^q(F'(x)+\gamma^2/2)x^{2/3}\mathrm{d}x, &q>0,\\
-\frac{\sigma^2}{5}(F'(0)+\gamma^2/2),&q=0. 
\end{cases}.$$
It is easy to see that $\varphi\in C^1(\mathbb{R}_+)$ is bounded and that for $q\geq0$,
\begin{equation}\label{eq:edo g}
5\varphi(q)+3q\varphi'(q)=-\sigma^2(F'(q)+\gamma^2/2).
\end{equation}
Consider now the auxiliary process $(x_t,v_t,r_t)_{t\geq0}$ solution to
\begin{equation}\label{eq:aux process}
  \left\{
    \begin{aligned}
        &\mathrm{d}x_t=v_t \mathrm{d}t , \\
        &\mathrm{d}v_t=-4r_tv_t\mathrm{d}t-3(2r_t^2+\varphi(x_t))x_t\mathrm{d}t+\sigma\mathrm{d}B_t,\\
        &\mathrm{d}r_t=(r_t^2+\varphi(x_t))\mathrm{d}t.
    \end{aligned}
\right.  
\end{equation}
Throughout this section, we fix the initial condition of the process $(r_t)_{t\geq0}$ at a constant $r_0>||\varphi||_\infty$. Let us first show that the process~\eqref{eq:aux process} is well-defined on a small enough time interval. 
\begin{lemma}[Well-posedness]\label{lem:non explo}
There exists $t_1>0$ only depending on $r_0$ such that the equation~\eqref{eq:aux process} admits a unique strong solution on $[0,t_1]$ and satisfies for all $t\in[0,t_1]$,
\begin{equation}\label{eq:ineq r_t R_t}
r_0\leq r_t\leq2r_0\,,\qquad 1\leq\mathrm{e}^{3R_t}\leq2, 
\end{equation}
where $R_t:=\int_0^tr_s\mathrm{d}s$.
\end{lemma}
\begin{proof}
Let $t_r>0$ be the first hitting time of $2r_0$ by the process $(r_t)_{t\geq0}$. From the equation above we deduce that for $t\in[0,t_r]$,
$$r_t\leq r_0+t(4r_0^2+||\varphi||_\infty).$$
In particular $t_r\geq t_1:=r_0/(4r_0^2+||\varphi||_\infty)$. In particular, the coefficients of~\eqref{eq:aux process} are locally Lipschitz continuous and grow at most linearly on the time interval $[0,t_1]$, hence the uniqueness of a strong solution. In addition, we deduce from the expression of $R_t$ that for all $t\in[0,t_1]$
$$1\leq\mathrm{e}^{3R_t}\leq\mathrm{e}^{6t_1r_0}.$$
Therefore, up to decreasing $t_1$, we can assume that~\eqref{eq:ineq r_t R_t} holds.
\end{proof}
Define now
\begin{equation}\label{eq:def expr Phi}
\Phi:(x,v,r)\in\mathbb{R}^3\mapsto 2rv^2+(6r^2x+3x\varphi(x))v+6r^3x^2.
\end{equation}
Let also
$$\tau^{\mathrm{aux}}_\partial:=\inf\{t>0:x_t\notin(0,1)\}.$$
The following lemma provides estimates on a given expectation appearing later in the proof of Proposition~\ref{prop:estimates first exit time Langevin}.
\begin{lemma}\label{lem:asymptot esp aux} Let $f$ be a positive function on $D$ such that there exist constants $a,b>0$ such that
\begin{equation}\label{eq:condition on f}
\forall (q,p)\in D,\qquad|f(q,p)|\leq \mathrm{e}^{a|p|^2+b}.
\end{equation}
For all $t\in(0,t_1]$, there exist constants $c_t,c'_t>0$ such that for all $(q,p)\in D$,
\begin{equation}\label{eq:estimate expectation E_xvr}
c'_tH(q,(p+3r_0q)/\sigma^{2/3})\leq\mathbb{E}_{(q,p,r_0)}\left[\mathbb{1}_{\tau^{\mathrm{aux}}_\partial>t}\mathrm{e}^{\Phi(x_t,v_t,r_t)/\sigma^2}f(x_t,v_t)\right]\leq c_th(q,(p+3r_0q)/\sigma^{2/3}).
\end{equation}
\end{lemma}}
\begin{proof}
\textcolor{black}{The equation~\eqref{eq:aux process} implies that
$$\mathrm{d}[\mathrm{e}^{R_t}(v_t+3r_tx_t)]=\sigma\mathrm{e}^{R_t}\mathrm{d}B_t.$$
Defining $u_t=\int_0^t\mathrm{e}^{2R_s}\mathrm{d}s$, we deduce the existence of a Brownian motion $(W_t)_{t\geq0}$ such that, $\mathbb{P}_{(q,p)}$ almost-surely, for all $t\in[0,t_1]$,
\begin{equation}\label{eq:v_t r_t x_t}
\mathrm{e}^{R_t}(v_t+3r_tx_t)=p+3r_0q+\sigma W_{u_t}.
\end{equation}
Let us define for $t\geq0$,
$$V_t:=p+3r_0q+\sigma W_t\,,\qquad Y_t:=q+t(p+3r_0q)+\sigma\int_0^tW_s\mathrm{d}s.$$
Notice that
\begin{align*}
\mathrm{d}(x_t\mathrm{e}^{3R_t})&=\mathrm{e}^{3R_t}(v_t+3r_tx_t)\mathrm{d}t\\
&=\mathrm{e}^{2R_t}V_{u_t}\mathrm{d}t\\
&=\mathrm{d}(Y_{u_t}).
\end{align*}
Given that the initial conditions coincide under $\mathbb{P}_{(q,p)}$ we deduce that for all $t\in[0,t_1]$,
\begin{equation}\label{eq:Y_u_t x_t}
Y_{u_t}=x_t\,\mathrm{e}^{3R_t}.
\end{equation}
Analyzing the exit event from $(0,1)$ of the process $(x_t)_{t\geq0}$ then amounts to studying the exit event for the integrated Brownian process $(Y_t)_{t\geq0}$ which was thoroughly studied in Section~\ref{sec:two-sided proba integrated brownian process}.}

\textcolor{black}{\textbf{Upper-bound}. Let us start with the proof of the upper-bound in~\eqref{eq:estimate expectation E_xvr}. Under the event $\tau_\partial^{\mathrm{aux}}>t$, it follows from~\eqref{eq:Y_u_t x_t} and~\eqref{eq:ineq r_t R_t} that
$$\forall s\in[0,u_t],\qquad\frac{Y_s}{2}\in(0,1).$$
In addition, the condition~\eqref{eq:condition on f} and the boundedness of $r_t$ by Lemma~\ref{lem:non explo} ensure the existence of a constant $C_1>0$ such that
$$f(x_t,v_t)\mathrm{e}^{\Phi(x_t,v_t,r_t)/\sigma^2}\leq\mathrm{e}^{C_1(1+|v_t|^2)}.$$
Furthermore, by~\eqref{eq:v_t r_t x_t},
$$|v_t|\leq 3r_tx_t+\mathrm{e}^{-R_t}|V_{u_t}|\leq 6r_0+|V_{u_t}|.$$
Therefore, there exists a constant $C_2>0$ such that
$$f(x_t,v_t)\mathrm{e}^{\Phi(x_t,v_t,r_t)/\sigma^2}\leq\mathrm{e}^{C_2(1+|V_{u_t}|^2)}.$$
Given that the process $(Y_s/2,V_s/2)_{s\geq0}$ shares the same law as the process $(\widehat{q}^{\sigma/2}_s,\widehat{p}_s^{\sigma/2})_{s\geq0}$ starting from $(q/2,(p+3r_0q)/2)$, we deduce that
$$\mathbb{E}_{(q,p,r_0)}\left[\mathbb{1}_{\tau^{\mathrm{aux}}_\partial>t}\mathrm{e}^{\Phi(x_t,v_t,r_t)/\sigma^2}f(x_t,v_t)\right]\leq\mathbb{E}_{(q/2,(p+3r_0q)/2)}\left[\mathbb{1}_{\widehat{\tau}^{\sigma/2}_\partial>u_t}\mathrm{e}^{C_2(1+|\widehat{p}_{u_t}^{\sigma/2}|^2)}\right].$$
Applying Proposition~\ref{prop:first exit estimates implies two sided} for $F\equiv0$, $\gamma=0$ along with Proposition~\ref{prop:two-sided proba final} ensures that the quantity above is bounded, up to a time-dependent constant, by
$$h\big(q/2,(p+3r_0q)/(2(\sigma/2)^{2/3})\big)=\frac{1}{2^{1/6}}h\big(q,(p+3r_0q)/\sigma^{2/3}\big),$$
hence the proof of the upper-bound.}

\textcolor{black}{\textbf{Lower-bound}. Differentiating~\eqref{eq:Y_u_t x_t}, we obtain that
\begin{align*}
x_t&=Y_{u_t}\mathrm{e}^{-3R_t}\\
v_t&=\mathrm{e}^{-R_t}V_{u_t}-3r_tY_{u_t}\mathrm{e}^{-3R_t}.
\end{align*}
Therefore, if $(Y_{u_t},V_{u_t})\in K$ for some compact set $K\subset D$, there exists a constant $c_t>0$ such that
$$\mathrm{e}^{\Phi(x_t,v_t,r_t)/\sigma^2}f(x_t,v_t)\geq c_t$$
since $\Phi,f$ are positive and continuous on $D$. In addition, if $Y_{u_s}\in(0,1)$ for all $s\in(0,t)$ then $x_s\in(0,1)$ for all $s\in(0,t)$. Consequently, there exists a constant $c_t>0$ such that for $t\in(0,t_1]$,
\begin{align*}
\mathbb{E}_{(q,p,r_0)}\left[\mathbb{1}_{\tau^{\mathrm{aux}}_\partial>t}\mathrm{e}^{\Phi(x_t,v_t,r_t)/\sigma^2}f(x_t,v_t)\right]&\geq c_t\mathbb{P}_{(q,p,r_0)}\left(\tau^{\mathrm{aux}}_\partial>t,(Y_{u_t},V_{u_t})\in K\right)\\
&\geq c_t\mathbb{P}_{(q,p+3r_0q)}(\widehat{\tau}_\partial^\sigma>u_t, (\widehat{q}^\sigma_{u_t},\widehat{p}^\sigma_{u_t})\in K)
\end{align*}
since the processes $(Y_s,V_s)_{s\geq0}$ and $(\widehat{q}^\sigma_s,\widehat{p}_s^\sigma)_{s\geq0}$ starting from $(q,p+3r_0q)$ share the same law, which concludes the proof by Lemma~\ref{lem:control in a compact} and Proposition~\ref{prop:two-sided proba final}.} 
\end{proof}
\begin{proof}[Proof of Proposition~\ref{prop:estimates first exit time Langevin}]

\textcolor{black}{\textbf{Step 1}: Let us first prove that for all $t\in(0,t_1]$,
\begin{equation}\label{eq:estimate proba step 2}
\mathbb{P}_{(q,p)}(\tau_\partial>t)\asymp\frac{1}{f(q,p)}\mathbb{E}_{(q,p)}\left[\mathbb{1}_{\widehat{\tau}^\sigma_\partial>t}f(\widehat{q}^\sigma_t,\widehat{p}^\sigma_t)\mathrm{e}^{-\frac{1}{\sigma^2}\int_0^t(\widehat{p}^\sigma_s)^2(F'(\widehat{q}^\sigma_s)+\gamma^2/2)}\mathrm{d}s\right].
\end{equation}
where $f(q,p)=\mathrm{e}^{\frac{1}{\sigma^2}(F(q)p-\gamma p^2/2)}$.}

\textcolor{black}{By the Cameron Martin formula,
\begin{equation}\label{eq:cameron martin exit proba}
\mathbb{P}_{(q,p)}(\tau_\partial>t)=\mathbb{E}_{(q,p)}\left[\mathbb{1}_{\widehat{\tau}^\sigma_\partial>t}Z_t^\sigma\right],
\end{equation} 
where $$Z_t^\sigma=\mathrm{exp}\left(\frac{1}{\sigma^2}\int_0^{t}\big(F(\widehat{q}^\sigma_s)-\gamma\widehat{p}^\sigma_s\big)\mathrm{d}\widehat{p}^\sigma_s-\frac{1}{2\sigma^2}\int_0^{t}\big(F(\widehat{q}^\sigma_s)-\gamma\widehat{p}^\sigma_s\big)^2\mathrm{d}s\right).$$ 
In addition, by integration by parts,
$$\int_0^{t}\big(F(\widehat{q}^\sigma_s)-\gamma\widehat{p}^\sigma_s\big)\mathrm{d}\widehat{p}^\sigma_s=F(\widehat{q}^\sigma_t)\widehat{p}^\sigma_t-F(q)p-\int_0^t(\widehat{p}^\sigma_s)^2F'(\widehat{q}^\sigma_s)\mathrm{d}s-\frac{\gamma}{2}(\widehat{p}^\sigma_t)^2+\frac{\gamma}{2}p^2+\gamma\sigma^2t.$$
Moreover, letting $c(q)=\int_0^qF(q)$, we obtain that
\begin{align*}
\frac{1}{2}\int_0^{t}\big(F(\widehat{q}^\sigma_s)-\gamma\widehat{p}^\sigma_s\big)^2\mathrm{d}s&=\frac{1}{2}\int_0^{t}\big(F(\widehat{q}^\sigma_s)^2-2\gamma\widehat{p}^\sigma_sF(\widehat{q}^\sigma_s)+\gamma^2(\widehat{p}^\sigma_s)^2\big)\mathrm{d}s\\
&=\frac{1}{2}\int_0^{t}F(\widehat{q}^\sigma_s)^2\mathrm{d}s-\gamma c(\widehat{q}^\sigma_t)+\gamma c(q)+\frac{\gamma^2}{2}\int_0^{t}(\widehat{p}^\sigma_s)^2\mathrm{d}s.
\end{align*}
Re-injecting in~\eqref{eq:cameron martin exit proba} we deduce~\eqref{eq:estimate proba step 2}.}

\textcolor{black}{\textbf{Step 2:} Let us now show the existence of constants $c_t,c'_t>0$ such that for all $t\in(0,t_1]$, for all $(q,p)\in D$,
\begin{equation}\label{eq:asymptot P tau_partial}
c'_t\frac{\mathrm{e}^{-\Phi(q,p,r_0)/\sigma^2}}{f(q,p)}H(q,(p+3r_0q)/\sigma^{2/3})\leq\mathbb{P}_{(q,p)}(\tau_\partial>t)\leq c_t\frac{\mathrm{e}^{-\Phi(q,p,r_0)/\sigma^2}}{f(q,p)}h(q,(p+3r_0q)/\sigma^{2/3}).
\end{equation}
Given~\eqref{eq:estimate proba step 2} and Lemma~\ref{lem:asymptot esp aux}, the proof of~\eqref{eq:asymptot P tau_partial} immediately follows if for all $t\in(0,t_1]$,
\begin{equation}\label{eq:asymptot esp}
\mathbb{E}_{(q,p)}\left[\mathbb{1}_{\widehat{\tau}^{\sigma}_\partial>t}\mathrm{e}^{-\int_0^t(\widehat{p}_s^\sigma)^2[F'(\widehat{q}^\sigma_s)+\gamma^2/2]\mathrm{d}s}f(\widehat{q}^\sigma_t,\widehat{p}^\sigma_t)\right]\asymp\mathrm{e}^{-\Phi(q,p,r_0)/\sigma^2}\mathbb{E}_{(q,p,r_0)}\left[\mathbb{1}_{\tau^{\mathrm{aux}}_\partial>t}\mathrm{e}^{\Phi(x_t,v_t,r_t)/\sigma^2}f(x_t,v_t)\right].
\end{equation}
Consider the process $(\widehat{r}^\sigma_t)_{t\geq0}$ solution to
$$\mathrm{d}\widehat{r}^\sigma_t=((\widehat{r}^\sigma_t)^2+\varphi(\widehat{q}^\sigma_t))\mathrm{d}t.$$
Since $-\partial_p\Phi(q,p,r)=-4rp-3(2r^2+\varphi(q))q$ corresponds to the drift term added to the Brownian motion in~\eqref{eq:aux process}, we obtain from the Cameron Martin formula that for all $(q,p)\in D$,
\begin{align*}
&\mathrm{e}^{-\Phi(q,p,r_0)/\sigma^2}\mathbb{E}_{(q,p,r_0)}\left[\mathbb{1}_{\tau^{\mathrm{aux}}_\partial>t}\mathrm{e}^{\Phi(x_t,v_t,r_t)/\sigma^2}f(x_t,v_t)\right]\\
&=\mathrm{e}^{-\Phi(q,p,r_0)/\sigma^2}\mathbb{E}_{(q,p,r_0)}\left[\mathbb{1}_{\widehat{\tau}^{\sigma}_\partial>t}\mathrm{e}^{\Phi(\widehat{q}^\sigma_t,\widehat{p}^\sigma_t,\widehat{r}^\sigma_t)/\sigma^2}\mathrm{e}^{-\frac{1}{\sigma^2}\int_0^t\partial_p\Phi(\widehat{q}^\sigma_s,\widehat{p}^\sigma_s,\widehat{r}^\sigma_s)\mathrm{d}\widehat{p}^\sigma_s}\mathrm{e}^{-\frac{1}{2\sigma^2}\int_0^t\partial_p\Phi(\widehat{q}^\sigma_s,\widehat{p}^\sigma_s,\widehat{r}^\sigma_s)^2\mathrm{d}s}f(\widehat{q}^\sigma_t,\widehat{p}^\sigma_t)\right].
\end{align*}
In addition, by It\^o's formula, $\mathbb{P}_{(q,p,r_0)}$ almost-surely,
\begin{align*}
&\Phi(\widehat{q}^\sigma_t,\widehat{p}^\sigma_t,\widehat{r}^\sigma_t)-\Phi(q,p,r_0)-\int_0^t\partial_p\Phi(\widehat{q}^\sigma_s,\widehat{p}^\sigma_s,\widehat{r}^\sigma_s)\mathrm{d}\widehat{p}^\sigma_s-\frac{1}{2}\int_0^t\partial_p\Phi(\widehat{q}^\sigma_s,\widehat{p}^\sigma_s,\widehat{r}^\sigma_s)^2\mathrm{d}s\\
&=\int_0^t\bigg[\widehat{p}_s^\sigma\partial_q\Phi(\widehat{q}_s^\sigma,\widehat{p}_s^\sigma,\widehat{r}^\sigma_s)+\frac{\sigma^2}{2}\partial^2_p\Phi(\widehat{q}_s^\sigma,\widehat{p}_s^\sigma,\widehat{r}^\sigma_s)+((\widehat{r}_s^\sigma)^2+\varphi(\widehat{q}^\sigma_s))\partial_r\Phi(\widehat{q}_s^\sigma,\widehat{p}_s^\sigma,\widehat{r}^\sigma_s)-\frac{1}{2}\partial_p\Phi(\widehat{q}^\sigma_s,\widehat{p}^\sigma_s,\widehat{r}^\sigma_s)^2\bigg]\mathrm{d}s\\
&=\int_0^t(\widehat{p}_s^\sigma)^2\big[5\varphi(\widehat{q}^\sigma_s)+3q\varphi'(\widehat{q}^\sigma_s)\big]\mathrm{d}s+\int_0^t\bigg[2\sigma^2\widehat{r}^\sigma_s-\frac{9}{2}(\widehat{q}_s^\sigma)^2 \varphi(\widehat{q}_s^\sigma)^2\bigg]\mathrm{d}s
\end{align*}
using the expression of $\Phi$ in~\eqref{eq:def expr Phi}. Furthermore, by~\eqref{eq:edo g}, the quantity above is equal to
$$-\sigma^2\int_0^t(\widehat{p}_s^\sigma)^2\big[F'(\widehat{q}^\sigma_s)+\gamma^2/2\big]\mathrm{d}s+\int_0^t\bigg[2\sigma^2\widehat{r}^\sigma_s-\frac{9}{2}(\widehat{q}_s^\sigma)^2 \varphi(\widehat{q}_s^\sigma)^2\bigg]\mathrm{d}s.$$
Since the trajectory $(\widehat{r}^\sigma_t)_{t\in[0,t_1]}$ is bounded by Lemma~\ref{lem:non explo}, the proof of~\eqref{eq:asymptot esp} follows immediately, hence~\eqref{eq:asymptot P tau_partial}.}

\textcolor{black}{\textbf{Step 3}. This last step focuses on the proof of~\eqref{eq:asymptot proba langevin}. If~\eqref{eq:asymptot proba langevin} does not hold there exists a sequence $(q_n,p_n)_{n\geq0}$ in $D$ such that the inequalities in~\eqref{eq:asymptot proba langevin} fail. Therefore, let us consider an arbitrary sequence $(q_n,p_n)_{n\geq0}$ in $D$ and let us show the existence of constants $c_t,c'_t>0$ such that for all $t\in(0,t_1]$, for all $n\geq0$,
\begin{equation}\label{eq:asymptot exit proba}
c'_t\mathrm{e}^{(A_+(q_n)p_n^++A_-(q_n)p_n^-)/\sigma^2}\mathrm{e}^{\gamma p_n^2/2\sigma^2}\leq\frac{\mathbb{P}_{(q_n,p_n)}(\tau_\partial>t)}{H(q_n,p_n/\sigma^{2/3})}\leq c_t\mathrm{e}^{(A_+(q_n)p_n^++A_-(q_n)p_n^-)/\sigma^2}\mathrm{e}^{\gamma p_n^2/2\sigma^2},
\end{equation}
where the functions $A_-, A_+$ are defined in~\eqref{eq:def A-} and~\eqref{eq:def A+}. The proof is divided into several substeps depending on the limiting properties of the sequence.}

\textcolor{black}{\textbf{Step 3.a}: Assume first that $\liminf_{n\rightarrow\infty}(p_n+3r_0q_n)/q_n^{1/3}=-\infty$. Up to taking an appropriate subsequence we can assume that $(p_n+3r_0q_n)/q_n^{1/3}\rightarrow-\infty$ when $n\rightarrow\infty$. We first prove that for $n\geq0$ large enough,
\begin{equation}\label{eq:min attainability H}
H(q_n,(p_n+3r_0q_n)/\sigma^{2/3})=h(q_n,(p_n+3r_0q_n)/\sigma^{2/3}).
\end{equation}
Using Remark~\ref{asymptotics g},
\begin{align}
h(q_n,(p_n+3r_0q_n)/\sigma^{2/3})&=q_n^{1/6}g((p_n+3r_0q_n)/(q_n^{1/3}\sigma^{2/3}))\nonumber\\
&\underset{n\rightarrow\infty}{\sim}\frac{\sigma^{5/3}}{6}\left(\frac{2}{9}\right)^{-5/6}\frac{q_n}{|p_n+3r_0q_n|^{5/2}}\mathrm{e}^{-2\vert p_n+3r_0q_n\vert^3/(9q_n\sigma^2)}.\label{eq:equiv h(q_n p_n)}
\end{align}
Remark~\ref{asymptotics g} also ensures that the function $z>0\mapsto g(z)/\sqrt{z}$ is bounded from below, hence the existence of a constant $c>0$ such that for $n\geq0$,
$$g(-(p_n+3r_0q_n)/(1-q_n)^{1/3}\sigma^{2/3})\geq c\sqrt{|p_n+3r_0q_n|}/(1-q_n)^{1/6}.$$
In particular, we deduce that
\begin{align*}
\limsup_{n\rightarrow\infty}\frac{h(q_n,(p_n+3r_0q_n)/\sigma^{2/3})}{h(1-q_n,-(p_n+3r_0q_n)/\sigma^{2/3})}&\leq\frac{\sigma^{5/3}}{6c}\left(\frac{2}{9}\right)^{-5/6}\lim_{n\rightarrow\infty}\frac{q_n}{|p_n+3r_0q_n|^3}\mathrm{e}^{-2\vert p_n+3r_0q_n\vert^3/(9q_n\sigma^2)}=0,
\end{align*}
hence~\eqref{eq:min attainability H} for $n\geq0$ large enough. Re-injecting in~\eqref{eq:asymptot P tau_partial}, we obtain that for $t\in(0,t_1]$,
\begin{equation}\label{eq:equiv exit proba q_n p_n first case}
c'_t\frac{\mathrm{e}^{-\Phi(q_n,p_n,r_0)/\sigma^2}}{f(q_n,p_n)}h(q_n,(p_n+3r_0q_n)/\sigma^{2/3})\leq\mathbb{P}_{(q_n,p_n)}(\tau_\partial>t)\leq c_t\frac{\mathrm{e}^{-\Phi(q_n,p_n,r_0)/\sigma^2}}{f(q_n,p_n)}h(q_n,(p_n+3r_0q_n)/\sigma^{2/3}).
\end{equation}
In addition, the asymptotics in~\eqref{eq:equiv h(q_n p_n)} and the expression of $\Phi$ in~\eqref{eq:def expr Phi} ensure that
\begin{align*}
\mathrm{e}^{-\Phi(q_n,p_n,r_0)/\sigma^2}h(q_n,(p_n+3r_0q_n)/\sigma^{2/3})&\underset{n\rightarrow\infty}{\sim}\frac{\sigma^{5/3}}{6}\left(\frac{2}{9}\right)^{-5/6}\frac{q_n}{|p_n|^{5/2}}\mathrm{e}^{-2\vert p_n\vert^3/(9q_n\sigma^2)}\mathrm{e}^{-3q_ng(q_n)p_n/\sigma^2}\\
&\underset{n\rightarrow\infty}{\sim}h(q_n,p_n/\sigma^{2/3})\mathrm{e}^{-3q_ng(q_n)p_n/\sigma^2}.
\end{align*}
The same proof as~\eqref{eq:min attainability H} ensures that $h(q_n,p_n/\sigma^{2/3})=H(q_n,p_n/\sigma^{2/3})$ for $n\geq0$ large enough. Therefore, we deduce from~\eqref{eq:equiv exit proba q_n p_n first case} that~\eqref{eq:asymptot exit proba} holds since $p_n^+=0$ for $n\geq0$ large enough and $A^-(q_n)=3q_ng(q_n)+F(q_n)$.}

\textcolor{black}{\textbf{Step 3.b}: Assume now that $\liminf_{n\rightarrow\infty}(-p_n+3r_0(1-q_n))/(1-q_n)^{1/3}=-\infty$. Up to taking an appropriate subsequence we can assume that $(-p_n+3r_0(1-q_n))/(1-q_n)^{1/3}\rightarrow-\infty$ when $n\rightarrow\infty$. Notice that this condition corresponds to the condition of the previous step 3.a up to the transformation, $(q,p)\mapsto (1-q,-p)$.}

\textcolor{black}{Let us define the process $(\widetilde{q_t},\widetilde{p}_t)_{t\geq0}=(1-q_t,-p_t)_{t\geq0}$ which is solution to, by~\eqref{eq:Langevin},
$$\left\{
    \begin{aligned}
        &\mathrm{d}\widetilde{q}_t=\widetilde{p}_t \mathrm{d}t , \\
        &\mathrm{d}\widetilde{p}_t=-F(1-\widetilde{q}_t)\mathrm{d}t-\gamma\widetilde{p}_t\mathrm{d}t+\sigma\mathrm{d}\widetilde{B}_t,\\
        &(\widetilde{q}_0,\widetilde{p}_0)=(1-q_0,-p_0),
    \end{aligned}
\right.  
$$
where $(\widetilde{B}_t)_{t\geq0}$ is a Brownian motion. In particular, if $\widetilde{\tau}_\partial$ is the first exit time from $D$ of the process $(\widetilde{q_t},\widetilde{p}_t)_{t\geq0}$, we obtain that
$$\mathbb{P}_{(q_n,p_n)}(\tau_\partial>t)=\mathbb{P}_{(1-q_n,-p_n)}(\widetilde{\tau}_\partial>t).$$
Notice also that the process $(\widetilde{q_t},\widetilde{p}_t)_{t\geq0}$ shares the same law as $(q_t,p_t)_{t\geq0}$ by replacing the force $F(q)$ with $-F(1-q)$. As a result, the inequality~\eqref{eq:asymptot P tau_partial} also holds for $\mathbb{P}_{(1-q_n,-p_n)}(\widetilde{\tau}_\partial>~t)$ by performing the following transformations to the left and right hand side terms of the inequality
$$q_n\rightarrow1-q_n,\qquad p_n\rightarrow-p_n,\qquad F(q_n)\rightarrow -F(1-q_n).$$
Given that $(1-q_n,-p_n)$ satisfies the same assumptions as in Step 3.a, we deduce the inequality~\eqref{eq:asymptot exit proba} since $p_n^-=0$ for $n\geq0$ large enough.}

\textcolor{black}{\textbf{Step 3.c}:  It remains now to consider the case where there exist constants $c,c'>0$ such that for all $n\geq0$,
\begin{equation}\label{eq:bounds p_n}
-c'q_n^{1/3}\leq p_n\leq c(1-q_n)^{1/3}.
\end{equation}
Given the boundedness of the sequence $(p_n)_{n\geq0}$, it is sufficient to show the existence of constants $c'_t,c_t>0$ such that for all $n\geq0$,
\begin{equation}\label{eq:proba exit bounds}
c'_tH(q_n,p_n/\sigma^{2/3})\leq\mathbb{P}_{(q_n,p_n)}(\tau_\partial>t)\leq c_tH(q_n,p_n/\sigma^{2/3}).
\end{equation}
Up to extracting subsequences, we can also assume that the sequences $(q_n)_{n\geq0}$, $(p_n)_{n\geq0}$ converge to limits denoted by $q_\infty$, $p_\infty$. The case $q_\infty\in(0,1)$ being obvious since $H(q_\infty,p_\infty)>0$, we only consider below the case $q_\infty\in\{0,1\}$.}

\textcolor{black}{\underline{Case 1}: Assume that $q_\infty=0$. Let us first prove the lower-bound in~\eqref{eq:proba exit bounds}. Since $h(1-q_n,-(p_n+3r_0q_n)/\sigma^{2/3})$ admits a positive limit, it is easy to deduce the existence of a constant $c''>0$ such that for $n\geq0$ large enough,
$$H(q_n,(p_n+3r_0q_n)/\sigma^{2/3})\geq c''h(q_n,(p_n+3r_0q_n)/\sigma^{2/3}).$$
Additionally, the increasing nature of the function $h$ in the $p$ coordinates also yields the following lower-bound
$$h(q_n,(p_n+3r_0q_n)/\sigma^{2/3})\geq h(q_n,p_n/\sigma^{2/3})\geq H(q_n,p_n/\sigma^{2/3})$$
which concludes the proof of the lower-bound by reinjecting in~\eqref{eq:asymptot P tau_partial}.}

\textcolor{black}{Let us now prove the upper-bound in~\eqref{eq:proba exit bounds}. First we show that the following quantity is bounded
$$\frac{h(q_n,(p_n+3r_0q_n)/\sigma^{2/3})}{h(q_n,p_n/\sigma^{2/3})}=\frac{g((p_n+3r_0q_n)/(q_n^{1/3}\sigma^{2/3}))}{g(p_n/(q_n^{1/3}\sigma^{2/3}))}.$$
If $p_n/q_n^{1/3}$ is bounded from above as well, the quantity above is also bounded. Otherwise, up to extracting a subsequence we can assume that $p_n/q_n^{1/3}\longrightarrow+\infty$ and the asymptotics of $g$ in Remark~\ref{asymptotics g} ensure that
$$\frac{h(q_n,(p_n+3r_0q_n)/\sigma^{2/3})}{h(q_n,p_n/\sigma^{2/3})}\underset{n\rightarrow\infty}{\sim}\sqrt{\frac{p_n+3r_0q_n}{p_n}}\underset{n\rightarrow\infty}{\longrightarrow}1.$$
Therefore, we deduce from~\eqref{eq:asymptot P tau_partial} the existence of a constant $c_t>0$ such that for $n\geq0$ large enough,
\begin{equation}\label{eq:proof upper bound 3.b.1}
\mathbb{P}_{(q_n,p_n)}(\tau_\partial>t)\leq c_th(q_n,p_n/\sigma^{2/3}).
\end{equation}
Since $h(1-q_n,-p_n/\sigma^{2/3})$ admits a positive limit, there exists also a constant $\beta>0$ such that
$$H(q_n,p_n/\sigma^{2/3})\geq\beta h(q_n,p_n/\sigma^{2/3}).$$
Re-injecting in~\eqref{eq:proof upper bound 3.b.1} concludes the proof of the upper-bound.}

\textcolor{black}{\underline{Case 2}: Consider now the case $q_\infty=1$. Applying the transformation detailed in Step 3.b to the inequality~\eqref{eq:asymptot P tau_partial}, we obtain constants $c_t,c'_t>0$ such that for $n\geq0$,
$$c'_tH(1-q_n,(-p_n+3r_0(1-q_n))/\sigma^{2/3})\leq\mathbb{P}_{(q_n,p_n)}(\tau_\partial>t)\leq c_th(1-q_n,(-p_n+3r_0(1-q_n))/\sigma^{2/3}).$$
The proof of~\eqref{eq:proba exit bounds} then follows the exact same computations detailed in Case 1 which achieves the proof of~\eqref{eq:asymptot proba langevin} for any $t\in(0,t_1]$. Its extension to any time $t>0$ can be shown by the exact same proof as Proposition~\ref{lower-bound 2} based on Lemma~\ref{lem:control in a compact}.}
\end{proof}

\section{Two-sided estimates, \textcolor{black}{boundary regularity} and long-time asymptotics}\label{sec:two-sided density}

The goal of this section is to prove the two-sided estimates, \textcolor{black}{the boundary regularity} and the long-time asymptotics satisfied by the absorbed transition density $\mathrm{p}^D_t$ using the first exit time probability estimates obtained in the previous section. In Section~\ref{sec:equivalence two sided}, we prove Proposition~\ref{prop:first exit estimates implies two sided} which therefore implies the two-sided estimates~\eqref{two-sided p_t^D} using Proposition~\ref{prop:estimates first exit time Langevin}. \textcolor{black}{In Section~\ref{sec:bdy regul}, we prove the boundary regularity statements appearing in Theorem~\ref{thm:two-sided p_t^D}, therefore concluding the proof of Theorem~\ref{thm:two-sided p_t^D}.} The long-time asymptotics of the transition density are investigated in Section~\ref{sec:long time asymptotics} where we prove Theorem~\ref{thm:long time}.
 
\subsection{Proof of Proposition~\ref{prop:first exit estimates implies two sided}}\label{sec:equivalence two sided}
Let $d\geq1$. Let $\mathcal{O}$ be a $\mathcal{C}^2$ bounded connected open set of $\mathbb{R}^d$ and let $D=\mathcal{O}\times\mathbb{R}^d$.  Let $(q_t,p_t)_{t\geq0}$ be the process in $\mathbb{R}^d$ satisfying 
\begin{equation}\label{eq:Langevin multi dim}
  \left\{
    \begin{aligned}
        &\mathrm{d}q_t=p_t \mathrm{d}t , \\
        &\mathrm{d}p_t=F(q_t)\mathrm{d}t-\gamma p_t\mathrm{d}t+\sigma\mathrm{d}B_t,
    \end{aligned}
\right.  
\end{equation}
where $F\in\mathcal{C}^\infty(\mathbb{R}^d)$, $\gamma\in\mathbb{R}$, $\sigma>0$. For $(x,y)\in D$, we denote by $\mathrm{p}_t^D(x,y)$ the transition density of $(q_t,p_t)_{t\geq0}$ absorbed outside of $D$ and $\tau_\partial$ its first exit time from $D$. Let us also define the process $(\tilde{q}_t,\tilde{p}_t)_{t\geq0}$ in $\mathbb{R}^d$ satisfying
\begin{equation}\label{eq:Langevin_adjoint}
  \left\{
    \begin{aligned}
        &\mathrm{d}\tilde{q}_t=-\tilde{p}_t \mathrm{d}t , \\
        &\mathrm{d}\tilde{p}_t=-F(\tilde{q}_t)\mathrm{d}t+\gamma \tilde{p}_t\mathrm{d}t+\sigma\mathrm{d}B_t.
    \end{aligned}
\right.  
\end{equation}
We also denote by $\tilde{\mathrm{p}}_t^D(x,y)$ the transition density of $(\tilde{q}_t,\tilde{p}_t)_{t\geq0}$ absorbed outside of $D$ and $\tilde{\tau}_\partial$ its first exit time from $D$. It was shown in~\cite[Theorem 6.2]{kFP} that for all $t>0$, $x,y\in D$,
\begin{equation}\label{eq:reversibility} 
    \mathrm{p}_t^D(x,y)=\mathrm{e}^{d\gamma t}\tilde{\mathrm{p}}_t^D(y,x).
\end{equation}

We are now ready to prove Proposition~\ref{prop:first exit estimates implies two sided}.
\begin{proof}[Proof of Proposition~\ref{prop:first exit estimates implies two sided}]

\textcolor{black}{Notice that the process $(\tilde{q}_t,-\tilde{p}_t)_{t\geq0}$ satisfies the equation
\begin{equation*}
  \left\{
    \begin{aligned}
        &\mathrm{d}\widetilde{q}_t=(-\widetilde{p}_t) \mathrm{d}t , \\
        &\mathrm{d}(-\widetilde{p}_t)=F(\widetilde{q}_t)\mathrm{d}t+\gamma(-\widetilde{p}_t)\mathrm{d}t+\sigma\mathrm{d}\widetilde{B}_t,
    \end{aligned}
\right.  
\end{equation*}
where $(\widetilde{B}_t)_{t\geq0}$ is a Brownian motion in $\mathbb{R}^n$. Therefore, the equation satisfied by the process $(\tilde{q}_t,-\tilde{p}_t)_{t\geq0}$ is the same as~\eqref{eq:Langevin multi dim} where the coefficient $-\gamma$ is replaced with $\gamma$. Therefore, by~\eqref{eq:hypothese prop 2.9}, for all $t>0$,
\begin{equation}\label{eq:two sided exit time proba process tilde}
\mathbb{P}_{(q,p)}(\tilde{\tau}_\partial>t)\asymp H_{F,-\gamma,\sigma}(q,-p).
\end{equation}}

In addition, by~\eqref{eq:hypothese prop 2.9} and Lemma~\ref{lem:control in a compact}, for all $t>0$, there exists a compact set $K_t\subset D$ such that
$$\beta_{t}:=\inf_{(q,p)\in D}\frac{\mathbb{P}_{(q,p)}((q_t,p_t)\in K_{t},\tau_\partial>t)}{\mathbb{P}_{(q,p)}(\tau_\partial>t)}>0.$$ 
In particular, we can choose the compact set $K_t$ above large enough such that it satisfies the following property:
\begin{equation}\label{def K_t}
    (q,p)\in K_t\Longleftrightarrow (q,-p)\in K_t.
\end{equation}
 
Since our assumptions do not depend on the sign of $\gamma$, we obtain the same control for the process $(\tilde{q}_t,-\tilde{p}_t)_{t\geq0}$
$$\tilde{\beta}_{t}:=\inf_{(q,p)\in D}\frac{\mathbb{P}_{(q,p)}((\tilde{q}_t,\tilde{p}_t)\in K_{t},\tilde{\tau}_\partial>t)}{\mathbb{P}_{(q,p)}(\tilde{\tau}_\partial>t)}>0.$$
We used here the fact that $\mathbb{P}_{(q,p)}((\tilde{q}_t,-\tilde{p}_t)\in K_{t},\tilde{\tau}_\partial>t)=\mathbb{P}_{(q,p)}((\tilde{q}_t,\tilde{p}_t)\in K_{t},\tilde{\tau}_\partial>t)$ by the symmetry of $K_t$.
We denote now by $\alpha_t$ the uniform upper-bound bound of $\mathrm{p}^D_t$ on $D\times D$ which follows from~\cite[Theorem 2.19]{kFP}. 

Let $(q,p),(q',p')\in D$. We start by proving the upper-bound. Using the Chapman-Kolmogorov relation and~\eqref{eq:reversibility}, one has
\begin{align*}
    \mathrm{p}^D_{t}(q,p,q',p')&=\iint_{D\times D}\mathrm{p}^D_{t/3}((q,p),y)\mathrm{p}^D_{t/3}(y,z)\mathrm{p}^D_{t/3}(z,(q',p'))\mathrm{d}y\mathrm{d}z\\
    &\leq\alpha_{t/3}\mathbb{P}_{(q,p)}(\tau_\partial>t/3)\int_{D}\mathrm{p}^D_{t/3}(z,(q',p'))\mathrm{d}z\\
    &=\alpha_{t/3}\mathrm{e}^{d\gamma t/3}\mathbb{P}_{(q,p)}(\tau_\partial>t/3)\int_{D}\tilde{\mathrm{p}}^D_{t/3}((q',p'),z)\mathrm{d}z\\
    &=\alpha_{t/3}\mathrm{e}^{d\gamma t/3}\mathbb{P}_{(q,p)}(\tau_\partial>t/3)\mathbb{P}_{(q',p')}(\tilde{\tau}_\partial>t/3).
\end{align*}
Therefore, we deduce from~\eqref{eq:two sided exit time proba process tilde} and~\eqref{eq:hypothese prop 2.9} the existence of a constant $\gamma_t>0$ such that for all $t>0$, for all $(q,p),(q',p')\in D$,
$$\mathrm{p}^D_{t}(q,p,q',p')\leq \gamma_tH_{F,\gamma,\sigma}(q,p)H_{F,-\gamma,\sigma}(q',-p').$$
Let us now prove the lower-bound for $\mathrm{p}^D_{t}$. Let 
$$M_{t/3}:=\inf_{(q,p),(q',p')\in K_{t/3}}\mathrm{p}^D_{t/3}(q,p,q',p')>0,$$
by the continuity and positivity of $\mathrm{p}^D_{t/3}$, see~\cite[Theorem 2.20]{kFP}. Again, by the Chapman-Kolmogorov relation, one has for all $t>0$, and for all $(q,p),(q',p')\in D$,
\begin{align*}
    &\mathrm{p}^D_{t}(q,p,q',p')\\
    &\geq\iint_{K_{t/3}\times K_{t/3}}\mathrm{p}^D_{t/3}((q,p),y)\mathrm{p}^D_{t/3}(y,z)\mathrm{p}^D_{t/3}(z,(q',p'))\mathrm{d}y\mathrm{d}z\\
    &\geq M_{t/3}\mathrm{e}^{d\gamma t/3}\mathbb{P}_{(q,p)}((q_{t/3},p_{t/3})\in K_{t/3},\tau_\partial>t/3)\mathbb{P}_{(q',p')}((\tilde{q}_{t/3},\tilde{p}_{t/3})\in K_{t/3},\tilde{\tau}_\partial>t/3),\numberthis\label{eq11}
\end{align*}
using~\eqref{eq:reversibility}. Furthermore, for all $(q,p)\in D$,
\begin{align*}
\mathbb{P}_{(q,p)}((q_{t/3},p_{t/3})\in K_{t/3},\tau_\partial>t/3)&\geq\beta_{t/3}\mathbb{P}_{(q,p)}(\tau_\partial>t/3)\\
\mathbb{P}_{(q',p')}((\tilde{q}_{t/3},\tilde{p}_{t/3})\in K_{t/3},\tilde{\tau}_\partial>t/3)&\geq\tilde{\beta}_{t/3}\mathbb{P}_{(q',p')}(\tilde{\tau}_\partial>t/3).
\end{align*}
 Re-injecting in~\eqref{eq11} and using~\eqref{eq:hypothese prop 2.9} and~\eqref{eq:two sided exit time proba process tilde}, we obtain the existence of a constant $\gamma'_t>0$ such that for all $t>0$, for all $(q,p),(q',p')\in D$,
$$\mathrm{p}^D_{t}(q,p,q',p')\geq\gamma'_tH_{F,\gamma,\sigma}(q,p)H_{F,-\gamma,\sigma}(q',-p').$$
Hence, for all $t>0$, $(q,p),(q',p')\in D$,
$$\gamma'_tH_{F,\gamma,\sigma}(q,p)H_{F,-\gamma,\sigma}(q',-p')\leq\mathrm{p}^D_{t}(q,p,q',p')\leq\gamma_tH_{F,\gamma,\sigma}(q,p)H_{F,-\gamma,\sigma}(q',-p'),$$
which concludes the proof of Proposition~\ref{prop:first exit estimates implies two sided}.
\end{proof}

\textcolor{black}{\subsection{Boundary regularity}\label{sec:bdy regul}
\begin{proposition}[Boundary regularity]\label{prop:bdy regul density} The function $\mathrm{p}^D$ belongs to $C^\infty(\mathbb{R}_+^*\times(\overline{D}\setminus\Gamma^0)\times(\overline{D}\setminus\Gamma^0))$. Moreover, for all $t>0$, for all $m,n\geq0$,
\begin{align}
\forall\;y\in\overline{D}\setminus\Gamma^0,\quad\forall\;(q,p)\in\Gamma^+,&\qquad\partial^{m}_{q}\partial^{n}_p\mathrm{p}^D_t((q,p),y)=0,\label{eq:backward deriv}\\
\forall\;x\in\overline{D}\setminus\Gamma^0,\quad\forall\;(q',p')\in\Gamma^-,&\qquad\partial^{m}_{q'}\partial^{n}_{p'}\mathrm{p}^D_t(x,(q',p'))=0.\label{eq:forward deriv}
\end{align}
\end{proposition}}
\begin{proof}\textcolor{black}{Fix $x_0,y_0\in\overline{D}\setminus\Gamma^0$ and let $\rho>0$ be small enough such that the balls $\mathrm{B}(x_0,\rho),\mathrm{B}(y_0,\rho)$ remain at a positive distance from $\Gamma^0$. Define
$$D_\rho:=D\cup\mathrm{B}(x_0,\rho)\cup\mathrm{B}(y_0,\rho).$$
Let $\tau^\rho_\partial$ be the first exit time from $D_\rho$ for the process~\eqref{eq:Langevin}.}

\textcolor{black}{\textbf{Step 1}: Let us first prove that for all $x\in D$, for any measurable set $A\subset D$,
\begin{equation}\label{eq:identification kernel D_rho}
\mathbb{P}_x((q_t,p_t)\in A,\tau^\rho_\partial>t)=\mathbb{P}_x((q_t,p_t)\in A,\tau_\partial>t).
\end{equation}
Since $D\subset D_\rho$, we deduce that
$$\mathbb{P}_x((q_t,p_t)\in A,\tau^\rho_\partial>t)-\mathbb{P}_x((q_t,p_t)\in A,\tau_\partial>t)=\mathbb{P}_x((q_t,p_t)\in A,\tau^\rho_\partial>t,\tau_\partial\leq t).$$
It follows from~\cite[Proposition 2.8]{kFP} that $(q_{\tau_\partial},p_{\tau_\partial})\in\Gamma^+$ almost-surely. Therefore, on the event $\{\tau_\partial\leq t<\tau_\partial^\rho\}$, necessarily $x_0$ or $y_0$ belongs to $\Gamma^+$ and $X_{\tau_\partial}\in\mathrm{B}(x_0,\rho)\cup\mathrm{B}(y_0,\rho)$.}

\textcolor{black}{To fix the notations, assume that $x_0=(q_0,p_0)\in\Gamma^+$ and that $X_{\tau_\partial}\in\mathrm{B}(x_0,\rho)$ which is at a positive distance from $\Gamma^0$. Therefore, there exists a constant $\eta>0$ such that, on the event $\{\tau_\partial\leq t\leq\tau_\partial^\rho\}$,
$$\frac{\mathrm{d}q_t}{\mathrm{d}t}=p_t\begin{cases}
\geq\eta, &\text{if }q_0=1,\\
\leq-\eta,&\text{if }q_0=0. 
\end{cases}$$
In particular, the process $(q_t,p_t)_{\tau_\partial\leq r\leq\tau_\partial^\rho}$ remains outside $D$, which contradicts the event $(q_t,p_t)\in A\subset D$, hence the equality~\eqref{eq:identification kernel D_rho}.}

\textcolor{black}{\textbf{Step 2}: Using the hypoellipticity of the operators $\mathcal{L}$, $\mathcal{L}^*$ defined in~\eqref{generateur Langevin qsd},~\eqref{generateur adjoint qsd} and applying~\cite[Theorem 4.3]{foldes2023method} on the open set $D_\rho$, the process $(q_t,p_t)$ absorbed outside of the open set $D_\rho$ admits a transition density $\mathrm{p}^{D_\rho}$ in $C^\infty(\mathbb{R}_+^*\times D_\rho\times D_\rho)$. As a result, reinjecting in~\eqref{eq:identification kernel D_rho}, we obtain that for all $t>0$, for all $x,y\in D$,
$$\mathrm{p}^{D_\rho}_t(x,y)=\mathrm{p}^D_t(x,y).$$
Since $\mathrm{p}^{D_\rho}\in C^\infty(\mathbb{R}_+^*\times D_\rho\times D_\rho)$ and $x_0,y_0$ are arbitrary points in $\overline{D}\setminus\Gamma^0$, we easily deduce that $\mathrm{p}^D$ can be extended as a function in $C^\infty(\mathbb{R}_+^*\times(\overline{D}\setminus\Gamma^0)\times(\overline{D}\setminus\Gamma^0))$.}

\textcolor{black}{\textbf{Step 3}: Let us now prove~\eqref{eq:backward deriv}. It follows from~\eqref{two-sided p_t^D} and the boundedness of $H_{F,-\gamma,\sigma}$ that for any $t>0$, there exists a constant $c_t>0$ such that for all $(q,p),(q',p')\in D$,
$$\mathrm{p}^D_t((q,p),(q',p'))\leq c_tH_{F,\gamma,\sigma}(q,p).$$
Fix $(q_0,p_0)\in\Gamma^+$. The expression of $H_{F,\gamma,\sigma}$ in~\eqref{eq:def H_alpha,beta,gamma,sigma} along with Remark~\ref{asymptotics g} ensure that
$$H_{F,\gamma,\sigma}(q,p)\underset{(q,p)\rightarrow(q_0,p_0)}{\sim}\frac{\sigma^{5/3}}{6}\left(\frac{2}{9}\right)^{-5/6}\frac{|q-q_0|}{|p_0|^{5/2}}\mathrm{e}^{-2\vert p\vert^3/(9|q-q_0|\sigma^2)}\mathrm{e}^{(A^+(q_0)p_0^++A^-(q_0)p_0^-)/\sigma^2}\mathrm{e}^{\gamma p_0^2/2\sigma^2}.$$
In particular, given the vanishing exponential dominant term, for any $m\geq0$,
$$\frac{1}{|q-q_0|^m}\mathrm{p}^D_t((q,p),(q',p'))\underset{(q,p)\rightarrow(q_0,p_0)}{\longrightarrow}0.$$
As a result, using the $C^\infty$ regularity of $\mathrm{p}^D$ obtained in Step 2, we deduce that for all $m\geq0$, for all $(q_0,p_0)\in\Gamma^+$, for all $(q',p')\in\overline{D}\setminus\Gamma^0$,
$$\partial_q^m\mathrm{p}^D_t((q_0,p_0),(q',p'))=0.$$
Differentiating the equality above in the $p_0$ coordinate, we also deduce that the mixed derivatives vanish, hence~\eqref{eq:backward deriv}. The proof of~\eqref{eq:forward deriv} is exactly similar.}
\end{proof}
\subsection{Long-time asymptotics}\label{sec:long time asymptotics}
The objective of this final section is to prove the long-time asymptotics in Theorem~\ref{thm:long time}.  The main ingredient of the proofs is the long-time asymptotics of the absorbed semigroup obtained in~\cite[Theorem 2.19]{QSD1} combined with the two-sided estimates from Theorem~\ref{thm:two-sided p_t^D}.

Let us start with the proof of the following proposition where we extend the long-time asymptotics of the absorbed semigroup obtained in~\cite[Theorem 2.19]{QSD1} to the set of functions $f\in\mathrm{L}_{\textcolor{black}{H_{F,\gamma,\sigma}}}$.
\begin{proposition}[Long-time asymptotics]\label{prop:long time semigroup}
There exists a spectral gap $\alpha>0$ such that for all $t_0>0$, there exists $C_{t_0}>0$ such that for all $t\geq t_0$, for all $f\in\mathrm{L}_{\textcolor{black}{H_{F,\gamma,\sigma}}}$, for all $(q,p)\in D$,
$$\left\vert\mathbb{E}_{(q,p)}\left[f(q_{t},p_{t})\mathbb{1}_{\tau_\partial>t}\right]-\mathrm{e}^{-\lambda_0 t}\frac{\int_D \psi f}{\int_D \phi \psi} \phi(q,p)\right\vert\leq C_{t_0}\Vert f\Vert_{\textcolor{black}{H_{F,\gamma,\sigma}}}\textcolor{black}{H_{F,\gamma,\sigma}(q,p)}\mathrm{e}^{-(\lambda_0+\alpha)t}$$
where $\phi$, $\psi$ are the eigenvectors defined in~\eqref{vp droite} and~\eqref{vp gauche}.
\end{proposition}
\begin{proof}[Proof of Proposition~\ref{prop:long time semigroup}]
We first consider the case $f\in\mathrm{L}^\infty(D)$. Let us fix $t_0>0$ and let $t>t_0$. By the Markov property at time $t_0/2$, one has
\begin{align*}
&\mathbb{E}_{(q,p)}\left[f(q_{t-t_0/2},p_{t-t_0/2})\mathbb{1}_{\tau_\partial>t-t_0/2}\right]-\mathrm{e}^{-\lambda_0 (t-t_0/2)}\frac{\int_D \psi f}{\int_D \phi \psi} \phi(q,p)\\
&=\mathbb{E}_{(q,p)}\left[\mathbb{1}_{\tau_\partial>t_0/2}\left(\mathbb{E}_{(q_{t_0/2},p_{t_0/2})}\left[f(q_{t-t_0},p_{t-t_0})\mathbb{1}_{\tau_\partial>t-t_0}\right]-\mathrm{e}^{-\lambda_0 (t-t_0)}\frac{\int_D \psi f}{\int_D \phi \psi} \phi(q_{t_0/2},p_{t_0/2})\right)\right]\numberthis\label{eq:markov long-time}   
\end{align*}
since $\mathbb{E}_{(q,p)}\left[\mathbb{1}_{\tau_\partial>t_0/2}\phi(q_{t_0/2},p_{t_0/2})\right]=\mathrm{e}^{-\lambda_0 t_0/2}\phi(q,p)$ by~\cite[Theorem 2.13]{QSD1}. Besides, it follows from the long-time asymptotics in~\cite[Theorem 2.19]{QSD1} that there exists $\alpha>0$ and $C>0$ such that almost-surely, for all $t>t_0$, 
$$\left\vert\mathbb{E}_{(q_{t_0/2},p_{t_0/2})}\left[f(q_{t-t_0},p_{t-t_0})\mathbb{1}_{\tau_\partial>t-t_0}\right]-\mathrm{e}^{-\lambda_0 (t-t_0)}\frac{\int_D \psi f}{\int_D \phi \psi} \phi(q_{t_0/2},p_{t_0/2})\right\vert\leq C\Vert f\Vert_{\infty}\mathrm{e}^{-(\lambda_0+\alpha)(t-t_0)}.$$
As a result, reinjecting in~\eqref{eq:markov long-time} we deduce that
\begin{align*}
    &\left\vert\mathbb{E}_{(q,p)}\left[f(q_{t-t_0/2},p_{t-t_0/2})\mathbb{1}_{\tau_\partial>t-t_0/2}\right]-\mathrm{e}^{-\lambda_0 (t-t_0/2)}\frac{\int_D \psi f}{\int_D \phi \psi} \phi(q,p)\right\vert\\
    &\leq C\mathbb{P}_{(q,p)}(\tau_\partial>t_0/2)\Vert f\Vert_{\infty}\mathrm{e}^{-(\lambda_0+\alpha)(t-t_0)}\\
    &\leq Cc_{t_0/2}\textcolor{black}{H_{F,\gamma,\sigma}(q,p)}\Vert f\Vert_{\infty}\mathrm{e}^{-(\lambda_0+\alpha)(t-t_0)}, 
\end{align*}
where the constant $c_{t_0/2}>0$ follows from Theorem~\ref{thm:two-sided p_t^D}.

Assume now that $f\in\mathrm{L}_{\textcolor{black}{H_{F,\gamma,\sigma}}}$, then since $H_{F,\gamma,\sigma}$ is bounded it follows from the upper-bound in Theorem~\ref{thm:two-sided p_t^D} that for all $s>0$, $P^D_s f\in\mathrm{L}^\infty(D)$ where the semigroup $(P^D_t)_{t\geq0}$ is defined in~\eqref{eq:def semigroup}. Applying the inequality above for $P^D_{t_0/2}f\in\mathrm{L}^\infty(D)$ instead of $f$ and using the Markov property in the left-hand side of the inequality, we obtain the existence of a constant $C_{t_0/2}>0$ such that for all $t>t_0$, for all $(q,p)\in D$,
\begin{equation}\label{ineq markov long-time}
    \left\vert\mathbb{E}_{(q,p)}\left[f(q_{t},p_{t})\mathbb{1}_{\tau_\partial>t}\right]-\mathrm{e}^{-\lambda_0 (t-t_0/2)}\frac{\int_D \psi P^D_{t_0/2} f}{\int_D \phi \psi} \phi(q,p)\right\vert\leq C_{t_0/2}\textcolor{black}{H_{F,\gamma,\sigma}(q,p)}\Vert P^D_{t_0/2} f\Vert_{\infty}\mathrm{e}^{-(\lambda_0+\alpha)t}.
\end{equation}
Moreover, by~\eqref{eq:reversibility} and using Fubini's theorem,
\begin{align*}
    \int_D \psi(q,p) P^D_{t_0/2} f(q,p)\mathrm{d}q\mathrm{d}p&=\mathrm{e}^{\gamma t_0/2}\int_D\left(\int_D\psi(q,p)\tilde{\mathrm{p}}^D_{t_0/2}(q',p',q,p)\mathrm{d}q\mathrm{d}p\right) f(q',p')\mathrm{d}q'\mathrm{d}p'\\
    &=\mathrm{e}^{\gamma t_0/2}\int_D\mathbb{E}_{(q',p')}\left[\psi(\tilde{q}_{t_0/2},\tilde{p}_{t_0/2})\mathbb{1}_{\tilde{\tau}_\partial>t_0/2}\right] f(q',p')\mathrm{d}q'\mathrm{d}p'\\
    &=\mathrm{e}^{-\lambda_0 t_0/2}\int_D\psi(q',p')f(q',p')\mathrm{d}q'\mathrm{d}p',\numberthis\label{eq:permut semigroupe psi}
\end{align*}
by~\cite[Theorem 2.13]{QSD1}. In addition, using the upper-bound in Theorem~\ref{thm:two-sided p_t^D}, there exists a constant $\alpha_{t_0/2}>0$ such that for all $(q,p)\in D$,
\begin{equation}\label{ineq semigroupe H}
\left\vert P^D_{t_0/2} f(q,p)\right\vert\leq \alpha_{t_0/2}\Vert H_{F,\gamma,\sigma}\Vert_\infty\Vert f\Vert_{H_{F,\gamma,\sigma}}.
\end{equation}
Re-injecting~\eqref{eq:permut semigroupe psi} and~\eqref{ineq semigroupe H} in~\eqref{ineq markov long-time} we obtain that for all $f\in\mathrm{L}_{\textcolor{black}{H_{F,\gamma,\sigma}}}$ and $t>t_0$,
$$\left\vert\mathbb{E}_{(q,p)}\left[f(q_{t},p_{t})\mathbb{1}_{\tau_\partial>t}\right]-\mathrm{e}^{-\lambda_0 t}\frac{\int_D \psi f}{\int_D \phi \psi} \phi(q,p)\right\vert\leq C_{t_0/2}\alpha_{t_0/2}\Vert H_{F,\gamma,\sigma}\Vert_{\infty}\Vert f\Vert_{H_{F,\gamma,\sigma}}H_{F,\gamma,\sigma}(q,p)\mathrm{e}^{-(\lambda_0+\alpha)t}$$
which concludes the proof.
\end{proof}
Let us now prove Theorem~\ref{thm:long time}. We first prove~\eqref{semigroup conditionnel ineq} before proving~\eqref{eq:long-time cv constants}. 
\begin{proof}[Proof of~\eqref{semigroup conditionnel ineq} in Theorem~\ref{thm:long time}] The idea of the proof is similar to the proof of~\cite[Theorem 2.22]{QSD1} but uses the sharper estimate provided by Proposition~\ref{prop:long time semigroup}.

Let $\theta$ be a probability measure on $D$. First notice that for $t>0$ and $(q,p)\in D$,
$$\mathbb{P}_{(q,p)}(\tau_\partial>t)\geq\mathbb{E}_{(q,p)}\left[\mathbb{1}_{\tau_\partial>t}\frac{\phi(q_{t},p_{t})}{\Vert\phi\Vert_{\infty}}\right]=\frac{\mathrm{e}^{-\lambda_0t}}{\Vert\phi\Vert_{\infty}}\phi(q,p).$$
Consequently,
\begin{equation}\label{eq:minoration P_theta}
    \mathbb{P}_{\theta}(\tau_\partial>t)=\int_D\mathbb{P}_{(q,p)}(\tau_\partial>t)\theta(\mathrm{d}q\mathrm{d}p)\geq \frac{\int_D\phi \mathrm{d}{\theta}}{\Vert\phi\Vert_{\infty}}\mathrm{e}^{-\lambda_0 t}.
\end{equation}
Let us now take $t_0>0$ and $t>t_0$. By Proposition~\ref{prop:long time semigroup} there exists $C_{t_0}>0$ and $\alpha>0$ such that for all $f\in\mathrm{L}_{H_{F,\gamma,\sigma}}$,
\begin{align*}
  \left\vert\frac{\mathbb{E}_{\theta}(f(q_t,p_t)\mathbb{1}_{\tau_\partial>t})}{\mathbb{P}_{\theta}(\tau_\partial>t)}-\int_D f\psi\right\vert&=\left\vert\frac{\int_D\left(\mathbb{E}_{(q,p)}(f(q_t,p_t)\mathbb{1}_{\tau_\partial>t})-\left(\int_D f\psi\right) \mathbb{P}_{\theta}(\tau_\partial>t)\right){\theta}(\mathrm{d}q\mathrm{d}p)}{\mathbb{P}_{\theta}(\tau_\partial>t)}\right\vert\\
  &\leq\int_D\frac{\left\vert\mathbb{E}_{(q,p)}(f(q_t,p_t)\mathbb{1}_{\tau_\partial>t})-\mathrm{e}^{-\lambda_0 t}\frac{\int_D \psi f}{\int_D \phi \psi} \phi(q,p)\right\vert}{\mathbb{P}_{\theta}(\tau_\partial>t)}\theta(\mathrm{d}q\mathrm{d}p)\\
  &+\left\vert\int_D \psi f\right\vert\int_D\frac{\left\vert\mathrm{e}^{-\lambda_0 t}\frac{\phi(q,p)}{\int_D \phi \psi} -  \mathbb{P}_{(q,p)}(\tau_\partial>t)\right\vert}{\mathbb{P}_{\theta}(\tau_\partial>t)}\theta(\mathrm{d}q\mathrm{d}p)\\
   &\leq\frac{C_{t_0}\mathrm{e}^{-(\lambda_0+\alpha )t}}{\mathbb{P}_{\theta}(\tau_\partial>t)}\left[\left(\int_D H_{F,\gamma,\sigma} \mathrm{d}{\theta}\right) \Vert f\Vert_{H_{F,\gamma,\sigma}}+\left(\int_D\psi\vert f\vert\right)\left(\int_DH_{F,\gamma,\sigma}\mathrm{d}{\theta}\right)\Vert 1\Vert_{H_{F,\gamma,\sigma}}\right]\\
  &\leq C_{t_0}\Vert \phi\Vert_{\infty}\frac{\int_DH_{F,\gamma,\sigma}\mathrm{d}{\theta}}{\int_D\phi\mathrm{d}\theta}\mathrm{e}^{-\alpha t} \left[ \Vert f\Vert_{H_{F,\gamma,\sigma}}+\left(\int_D\psi\vert f\vert\right)\Vert 1\Vert_{H_{F,\gamma,\sigma}}\right]\numberthis\label{eq:majo int psi f}
\end{align*}
using~\eqref{eq:minoration P_theta}. Moreover, by Corollary~\ref{coroll phi psi} there exists a constant $c>0$ independent of $f\in\mathrm{L}_{H_{F,\gamma,\sigma}}(D)$ such that
$$\int_D\psi\vert f\vert\leq c\Vert f\Vert_{H_{F,\gamma,\sigma}},\qquad\frac{\int_DH_{F,\gamma,\sigma}\mathrm{d}{\theta}}{\int_D\phi\mathrm{d}\theta}\leq c.$$  
Reinjecting the inequalities above in~\eqref{eq:majo int psi f} concludes the proof.
\end{proof}
Let us conclude this work with the proof of~\eqref{eq:long-time cv constants}.
\begin{proof}[Proof of~\eqref{eq:long-time cv constants} in Theorem~\ref{thm:long time}]
This proof follows from applying Proposition~\ref{prop:long time semigroup} to the absorbed transition density $(q,p)\in D\mapsto\mathrm{p}^D_s(q,p,q',p')$ for $s>0$ and $(q',p')\in D$. The fact that such function is in $\mathrm{L}_{H_{F,\gamma,\sigma}}$ for all $(q',p')\in D$ follows from the two-sided estimates~\eqref{two-sided p_t^D} and the fact that $H_{F,\gamma,\sigma}$ is bounded and integrable, see Remark~\ref{rmk: prop H}. As a result, the Chapman-Kolmogorov relation ensures that for all $(q,p),(q',p')\in D$, $t\geq t_0$, $s>0$,
\begin{equation}\label{eq:apply long-time p^D_t}
    \left\vert\mathrm{p}^D_{t+s}(q,p,q',p')-\mathrm{e}^{-\lambda_0 t}\frac{\int_D \psi \mathrm{p}^D_s(\cdot,\cdot,q',p')}{\int_D \phi \psi} \phi(q,p)\right\vert\leq C_{t_0}H_{F,\gamma,\sigma}(q,p)\Vert \mathrm{p}^D_s(\cdot,\cdot,q',p')\Vert_{H_{F,\gamma,\sigma}}\mathrm{e}^{-(\lambda_0+\alpha)t}.
\end{equation}
Furthermore, by~\eqref{eq:reversibility},
\begin{align*}
\int_D \psi(q,p) \mathrm{p}^D_s(q,p,q',p')\mathrm{d}q\mathrm{d}p&=\mathrm{e}^{\gamma s}\int_D \psi(q,p) \tilde{\mathrm{p}}^D_s(q',p',q,p)\mathrm{d}q\mathrm{d}p\\ 
&=\mathrm{e}^{\gamma s}\mathbb{E}_{(q',p')}\left[\psi(\tilde{q}_{s},\tilde{p}_{s})\mathbb{1}_{\tilde{\tau}_\partial>s}\right]\\ 
&=\psi(q',p')\mathrm{e}^{-\lambda_0s}.
\end{align*} 
 In addition, one has from Corollary~\ref{coroll phi psi} the existence of constants $c_1,c_2>0$ such that for all $(q',p')\in D$, for all $s>0$,
\begin{align*}
    \Vert \mathrm{p}^D_s(\cdot,\cdot,q',p')\Vert_{H_{F,\gamma,\sigma}}&=\int_D\mathrm{p}^D_s(q,p,q',p')H_{F,-\gamma,\sigma}(q,-p)\mathrm{d}q\mathrm{d}p\\
    &\leq c_1\mathrm{e}^{\gamma s}\int_D\tilde{\mathrm{p}}^D_s(q',p',q,p)\psi(q,p)\mathrm{d}q\mathrm{d}p\\
    &=c_1\mathrm{e}^{-\lambda_0s}\psi(q',p')\\
    &\leq c_1c_2\mathrm{e}^{-\lambda_0s}H_{F,-\gamma,\sigma}(q',-p').
\end{align*}
As a result, reinjecting in~\eqref{eq:apply long-time p^D_t} one has the existence of a constant $C>0$ such that for all $s>0$, $t>t_0$,
$$\left\vert\mathrm{p}^D_{t+s}(q,p,q',p')-\mathrm{e}^{-\lambda_0 (t+s)}\frac{\phi(q,p)\psi(q',p')}{\int_D \phi \psi} \right\vert\leq C\mathrm{e}^{\alpha s}H_{F,\gamma,\sigma}(q,p)H_{F,-\gamma,\sigma}(q',-p')\mathrm{e}^{-(\lambda_0+\alpha)(t+s)}.$$
Therefore, taking the limit $t+s\rightarrow\infty$ for fixed $s>0$ ensures~\eqref{eq:long-time cv constants}, which concludes the proof of Theorem~\ref{thm:two-sided p_t^D}.
\end{proof}


\bibliographystyle{amsplain}
\bibliography{biblio}










\begin{acks}
Part of this work was supported by Samsung Science and Technology Foundation. The author would like to thank Nicolas Champagnat and Denis Villemonais for discussions leading to this work. The author also thanks Tony Lelièvre and Julien Reygner for corrections and discussions during early drafts of this work.
\end{acks}

\end{document}

%% file: plot_decomposition.tex
        \draw (2,0.5) node[anchor=north west] {\footnotesize{$(1,0)$}};
        \draw (-1,0.5) node[anchor=north east] {\footnotesize{$(0,0)$}};
        
        
        \draw[thick,red] (-1,0.5) -- (-1,2.5);
        \draw[thick,red] (2,0.5) -- (2,-1.5); 
        \draw[thick,blue] (-1,0.5) -- (-1,-1.5);
        \draw[thick,blue] (2,0.5) -- (2,2.5); 
        \draw[thick] (-1.9,0.5) -- (-1,0.5);
        \draw[thick,->] (2,0.5) -- (2.9,0.5)node[anchor=north west] {$q$};
        \draw[thick,->] (-1,2.5) -- (-1,2.5)node[anchor=north west] {$p$};
        
        \draw[scale=0.5, domain=-2:4, smooth, variable=\x, black] plot ({\x}, {1+(4-\x)^(1/3)});
        \draw[scale=0.5, domain=-2:4, smooth, variable=\x, black] plot ({\x}, {1-(\x+2)^(1/3)});
        \node[thick] at (0.5,0.5) {\textbf{2}};
        \node[thick] at (0.5,-1) {\textbf{3}};
        \node[thick] at (0.5,2) {\textbf{1}};
         \draw[black!40!blue,fill=black!40!blue] (-1.5,-1.0) node[anchor=south east] {$\Gamma^+$};
         \draw[black!40!blue,fill=black!40!blue] (3,2.0) node[anchor=south west] {$\Gamma^+$};
         \draw[black!40!red,fill=red] (-1.5,2.0) node[anchor=south east] {$\Gamma^-$};
         \draw[black!40!red,fill=red] (2.5,-1) node[anchor=south west] {$\Gamma^-$};
         \draw[black!40!green,fill=black!40!green] (-1,0.5) circle (.3ex)node[anchor=south east] {$\Gamma^0$};
         \draw[black!40!green,fill=black!40!green] (2,0.5) circle (.3ex)node[anchor=south west] {$\Gamma^0$};
        \draw[thick,red,->] (-1,1) -- (-0.8,1);
        \draw[thick,red,->] (-1,1.5) -- (-0.6,1.5);
        \draw[thick,red,->] (-1,2) -- (-0.4,2);
        \draw[thick,red,->] (-1,2.5) -- (-0.2,2.5); 
        \draw[thick,red,->] (2,0) -- (1.8,0);
        \draw[thick,red,->] (2,-0.5) -- (1.6,-0.5);
        \draw[thick,red,->] (2,-1) -- (1.4,-1);
        \draw[thick,red,->] (2,-1.5) -- (1.2,-1.5);
        \draw[thick,blue,->] (2,1) -- (2.2,1);
        \draw[thick,blue,->] (2,1.5) -- (2.4,1.5);
        \draw[thick,blue,->] (2,2) -- (2.6,2);
        \draw[thick,blue,->] (2,2.5) -- (2.8,2.5);
        \draw[thick,blue,->] (-1,0) -- (-1.2,0);
        \draw[thick,blue,->] (-1,-0.5) -- (-1.4,-0.5);
        \draw[thick,blue,->] (-1,-1) -- (-1.6,-1);
        \draw[thick,blue,->] (-1,-1.5) -- (-1.8,-1.5);

%% file: biblio.bib
@article {VilMel,
    AUTHOR = {M\'{e}l\'{e}ard, S. and Villemonais, D.},
     TITLE = {Quasi-stationary distributions and population processes},
   JOURNAL = {Probab. Surv.},
  FJOURNAL = {Probability Surveys},
    VOLUME = {9},
      YEAR = {2012},
     PAGES = {340--410}, 
}

@article{foldes2023method,
  title={The method of stochastic characteristics for linear second-order hypoelliptic equations},
  author={F{\"o}ldes, J. and Herzog, D. P.},
  journal={Probability Surveys},
  volume={20},
  pages={113--169},
  year={2023},
  publisher={The Institute of Mathematical Statistics and the Bernoulli Society}
}

@article{CattColMelMart,
  title={Quasi-stationary distributions and diffusion models in population dynamics},
  author={Cattiaux, P. and Collet, P. and Lambert, A. and Mart{\'\i}nez, S. and M{\'e}l{\'e}ard, S. and San Mart{\'\i}n, J.},
  journal={The Annals of Probability},
  volume={37},
  number={5},
  pages={1926--1969},
  year={2009},
  publisher={Institute of Mathematical Statistics}
}

@article{isozaki1994,
author = "Isozaki, Y. and Watanabe, S.",
doi = "10.3792/pjaa.70.271",
fjournal = "Proceedings of the Japan Academy, Series A, Mathematical Sciences",
journal = "Proc. Japan Acad. Ser. A Math. Sci.",
number = "9",
pages = "271--276",
publisher = "The Japan Academy",
title = "An asymptotic formula for the {K}olmogorov diffusion and a refinement of {S}inai's estimates for the integral of {B}rownian motion",
url = "https://doi.org/10.3792/pjaa.70.271",
volume = "70",
year = "1994"
}

@incollection{champagnat2018criteria,
  title={Criteria for exponential convergence to quasi-stationary distributions and applications to multi-dimensional diffusions},
  author={Champagnat, Nicolas and Coulibaly-Pasquier, Kol{\'e}h{\`e} Abdoulaye and Villemonais, Denis},
  booktitle={S{\'e}minaire de Probabilit{\'e}s XLIX},
  pages={165--182},
  year={2018},
  publisher={Springer}
}

@article{sinai,
  title={Distribution of some functionals of the integral of a random walk},
  author={Sinai, Y. G.},
  journal={Theoretical and Mathematical Physics},
  volume={90},
  number={3},
  pages={219--241},
  year={1992},
  publisher={Springer}
}

@article{lachal1,
  title={On the first passage time for integrated {B}rownian motion},
  author={Lachal, A.},
  journal={Ann. IHP, Sect. B},
  volume={27},
  number={3},
  pages={385--405},
  year={1991}
}

@article{benaim,
      title={Degenerate processes killed at the boundary of a domain}, 
      author={M. Benaïm and N. Champagnat and W. Oçafrain and D. Villemonais},
      year={2021},
      journal = {arXiv e-prints},
      eprint={2103.08534},
      archivePrefix={arXiv},
      primaryClass={math.PR}
}

@ARTICLE{V,
       author = {{Champagnat}, N. and {Villemonais}, D.},
        title = "{General criteria for the study of quasi-stationarity}",
      journal = {arXiv e-prints},
         year = "2017",
        month = "Dec",
          eid = {arXiv:1712.08092},
        pages = {arXiv:1712.08092},
archivePrefix = {arXiv},
       eprint = {1712.08092},
 primaryClass = {math.PR},
       adsurl = {https://ui.adsabs.harvard.edu/abs/2017arXiv171208092C}
}

@book{Collet,
    AUTHOR = {Collet, P. and Mart\'{\i}nez, S. and San Mart\'{\i}n, J.},
     TITLE = {Quasi-stationary distributions},
    SERIES = {Probability and its Applications (New York)},
      NOTE = {Markov chains, diffusions and dynamical systems},
 PUBLISHER = {Springer, Heidelberg},
      YEAR = {2013},
     PAGES = {xvi+280}, 
}

@article {GQZ,
    AUTHOR = {Gong, G. L. and Qian, M. P. and Zhao, Z. X.},
     TITLE = {Killed diffusions and their conditioning},
   JOURNAL = {Probab. Theory Related Fields},
  FJOURNAL = {Probability Theory and Related Fields},
    VOLUME = {80},
      YEAR = {1988}, 
     PAGES = {151--167}, 
}

@article {KnobPart,
    AUTHOR = {Knobloch, R. and Partzsch, L.},
     TITLE = {Uniform conditional ergodicity and intrinsic
              ultracontractivity},
   JOURNAL = {Potential Anal.},
  FJOURNAL = {Potential Analysis. An International Journal Devoted to the
              Interactions between Potential Theory, Probability Theory,
              Geometry and Functional Analysis},
    VOLUME = {33},
      YEAR = {2010},
    NUMBER = {2},
     PAGES = {107--136}, 
}

@article{GuiNectoux,
  TITLE = {{Quasi-stationary distribution for strongly Feller Markov processes by Lyapunov functions and applications to hypoelliptic Hamiltonian systems}},
  AUTHOR = {Guillin, A. and Nectoux, B. and Wu, L.}, 
  YEAR = {2020},
  JOURNAL = {https://hal.archives-ouvertes.fr/hal-03068461/}
}

@book{Olver,
  title={Asymptotics and special functions},
  author={Olver, F.},
  year={1997},
  publisher={CRC Press}
}

@article{QSD1,
title = {Quasi-stationary distribution for the {L}angevin process in cylindrical domains, Part I: Existence, uniqueness and long-time convergence},
journal = {Stochastic Processes and their Applications},
volume = {144},
pages = {173-201},
year = {2022}, 
author = {T. Lelièvre and M. Ramil and J. Reygner},  
}

@article {Vel,
    AUTHOR = {Hwang, H. J. and Jang, J. and Vel\'{a}zquez, J. J. L.},
     TITLE = {The {F}okker-{P}lanck equation with absorbing boundary
              conditions},
   JOURNAL = {Arch. Ration. Mech. Anal.},
  FJOURNAL = {Archive for Rational Mechanics and Analysis},
    VOLUME = {214},
      YEAR = {2014},
    NUMBER = {1},
     PAGES = {183--233}, 
}

@ARTICLE{kFP,
       author = {T. Lelièvre and M. Ramil and J. Reygner},
        title = "{A probabilistic study of the kinetic {F}okker-{P}lanck equation in cylindrical domains}",
      journal = {Journal of Evolution Equations}, 
         year = "2022",
         VOLUME = {22:38},
}

@article{Gorkov,
author = {Yu.~P.~Gor'kov},
title = {A formula for the solution of a boundary value problem for the stationary equation of Brownian motion},
volume = {223},
journal = {Dokl. Akad. Nauk SSSR},
number = {3}, 
pages = {525--528},
year = {1975} 
}

@article{kim2006two,
  title={Two-sided estimates on the density of Brownian motion with singular drift},
  author={Kim, P. and Song, R.},
  journal={Illinois Journal of Mathematics},
  volume={50},
  number={1-4},
  pages={635--688},
  year={2006},
  publisher={Duke University Press}
}

@article{gruter1982green,
  title={The {G}reen function for uniformly elliptic equations},
  author={Gr{\"u}ter, M. and Widman, K.-O.},
  journal={Manuscripta mathematica},
  volume={37},
  number={3},
  pages={303--342},
  year={1982},
  publisher={Springer}
}

@article{Lefebvre,   
 author = {M. Lefebvre},
 journal = {SIAM Journal on Applied Mathematics},
 number = {5},
 pages = {1514--1523},
 publisher = {Society for Industrial and Applied Mathematics},
 title = {First-Passage Densities of a Two-Dimensional Process},
 volume = {49},
 year = {1989}
}

@article{Goldman,
author = {M. Goldman},
title = {{On the First Passage of the Integrated Wiener Process}},
volume = {42},
journal = {The Annals of Mathematical Statistics},
number = {6},
publisher = {Institute of Mathematical Statistics},
pages = {2150 -- 2155},
year = {1971},
doi = {10.1214/aoms/1177693084},
URL = {https://doi.org/10.1214/aoms/1177693084}
}

@article{McKean,
author = {H. P. McKean Jr.},
title = {{A winding problem for a resonator driven by a white noise}},
volume = {2},
journal = {Journal of Mathematics of Kyoto University},
number = {2},
publisher = {Duke University Press},
pages = {227 -- 235},
year = {1962},
doi = {10.1215/kjm/1250524936},
URL = {https://doi.org/10.1215/kjm/1250524936}
}

@article{GP,
  title={Integrated {B}rownian motion, conditioned to be positive},
  author={Groeneboom, P. and Jongbloed, G. and Wellner, J. A.},
  journal={The Annals of Probability},
  volume={27},
  number={3},
  pages={1283--1303},
  year={1999},
  publisher={Institute of Mathematical Statistics}
}

@article{lachal2,
  title={Temps de sortie d’un intervalle born{\'e} pour l’int{\'e}grale du mouvement {B}rownien},
  author={Lachal, A.},
  journal={Comptes Rendus de l'Acad{\'e}mie des Sciences-Series I-Mathematics},
  volume={324},
  number={5},
  pages={559--564},
  year={1997},
  publisher={Elsevier}
}
